\documentclass[british,12pt]{article}

\usepackage{babel}

\usepackage{multicol}
\usepackage[all]{xy}
\usepackage{bbm,enumerate}

\usepackage{amsmath,amssymb,amsthm}

\newtheorem{theorem}{Theorem}[section]
\newtheorem{lemma}[theorem]{Lemma}
\newtheorem{proposition}[theorem]{Proposition}
\newtheorem{corollary}[theorem]{Corollary}

\theoremstyle{definition}
\newtheorem{remark}[theorem]{Remark}

\newcommand{\torus}{\mathbb{T}^2}
\newcommand{\klein}{\mathbb{K}^2}
\newcommand{\z}{\mathbb{Z}}
\newcommand{\zsz}{\mathbb{Z} \oplus \mathbb{Z}}
\newcommand{\zsdz}{\mathbb{Z} \rtimes \mathbb{Z}}
\newcommand{\ztwo}{\mathbb{Z}_2}
\newcommand{\rtwo}{\mathbb{R}^2}
\newcommand{\id}{\boldsymbol{1}}

\renewcommand{\hom}{{\rm Hom}}
\newcommand{\closure}{\overline{ \left\langle \sigma^2 \right\rangle}}
\newcommand{\gsigma}{\overline{\left\langle \sigma^2 \right\rangle}}
\newcommand{\gsigmab}{\overline{\left\langle \sigma^2 \right\rangle}_{\text{Ab}}}
\newcommand{\ab}{{\text{Ab}}}

\renewcommand{\ker}[1]{\ensuremath{\operatorname{\text{Ker}}\left({#1}\right)}}

\usepackage[x11names]{xcolor}

\usepackage[%pdftex,                %
bookmarks=true,
breaklinks=true,
bookmarksnumbered = true,%     % Signets numÃ©rotÃ©s
colorlinks= true,%     % Liens en couleur
urlcolor= green,
anchorcolor = yellow,
citecolor=blue,%  % Couleur des liens externes
]{hyperref}                   % Utilisation de HyperTeX

\makeatletter

\renewcommand{\p@enumii}{}
\makeatother

\allowdisplaybreaks

\begin{document}

\title{The Borsuk-Ulam property for homotopy classes of maps between the torus and the Klein bottle}

\author{DACIBERG LIMA GON\c{C}ALVES
~\footnote{Departamento de Matem\'atica, IME, Universidade de S\~ao Paulo, Caixa Postal 66281, Ag.\ Cidade de S\~ao Paulo, CEP: 05314-970, S\~ao Paulo, SP, Brazil. 
e-mail: \texttt{dlgoncal@ime.usp.br}}
\and
JOHN GUASCHI
~\footnote{Normandie Univ., UNICAEN, CNRS, Laboratoire de Mathématiques Nicolas Oresme UMR CNRS~\textup{6139}, 14000 Caen, France.
e-mail: \texttt{john.guaschi@unicaen.fr}}
\and
VINICIUS CASTELUBER LAASS
~\footnote{Departamento de Matemática, IME, Universidade Federal da Bahia, Av.\ Adhemar de Barros, S/N Ondina CEP: 40170-110, Salvador, BA, Brazil. 
e-mail: \texttt{vinicius.laass@ufba.br}} 
%\thanks{}
}

\date{17th November 2019}

\maketitle

\begin{abstract}%
Let $M$ be a topological space that admits a free involution $\tau$, and let $N$ be a topological space. A homotopy class $\beta \in [ M,N ]$ is said to have {\it the Borsuk-Ulam property with respect to $\tau$} if for every representative map $f: M \to N$ of $\beta$, there exists a point $x \in M$ such that $f(\tau(x))= f(x)$. In this paper, we determine the homotopy classes of maps from the $2$-torus $\torus$ to the Klein bottle $\klein$ that possess the Borsuk-Ulam property  with respect to a free involution $\tau_1$ of $\torus$ for which the orbit space is $\torus$. Our results are given in terms of a certain family of homomorphisms involving the fundamental groups of $\torus$ and $\klein$.
\end{abstract}

%%%%%%%%%%%%%%%%%%%%%%%%%%%%%%%%%%%%%%%%%%%%%%%%%%%%%%%%%%%%%%%%%%%%%%%%%%%%%%%%%%%%%%%%%%%%%%%%%%%%%%%%%%%%%%%%%%%%%%%%%%%%%%%%%%%%%%%%%%%
% INTRODUCTION
%%%%%%%%%%%%%%%%%%%%%%%%%%%%%%%%%%%%%%%%%%%%%%%%%%%%%%%%%%%%%%%%%%%%%%%%%%%%%%%%%%%%%%%%%%%%%%%%%%%%%%%%%%%%%%%%%%%%%%%%%%%%%%%%%%%%%%%%%%%

\section{Introduction}\label{sec:introduction}

In the early twentieth century, St.~Ulam conjectured that if $ f : \mathbb{S}^n \to \mathbb{R}^n$ is a continuous map, there exists $ x \in \mathbb{S}^n$ such that $f(A (x))=f (x)$, where $ A: \mathbb{S}^n \to \mathbb{S}^n$ is the antipodal map. The confirmation of this result by K.~Borsuk in 1933~\cite{Borsuk}, known as the Borsuk-Ulam theorem, was the beginning of what it now referred to as {\it Borsuk-Ulam type theorems} or the {\it Borsuk-Ulam property}. More information about the history and some applications of the Borsuk-Ulam theorem may be found in~\cite{Mato}, for example.

One possible generalisation of the classical Borsuk-Ulam theorem is to subs\-titute $\mathbb{S}^n$ and $\mathbb{R}^n$ by other spaces, and to replace the antipodal map by a free involution. A natural question is the following: does every continuous map collapse an orbit of the involution? More precisely, given topological spaces $M$ and $N$ such that $M$ admits a free involution $\tau$, we say that the triple $(M,\tau ; N)$ {\it has the Borsuk-Ulam property} if for every continuous map $f: M \to N$, there exists a point $x \in M$ such that $f(\tau(x))=f(x)$. For the cases where $M$ is a compact surface without boundary admitting a free involution $\tau$ and $N$ is either $\mathbb{R}^2$ or a compact surface without boundary, the triples $(M,\tau ; N)$ that have the Borsuk-Ulam property have been classified (see~\cite{Gon} and~\cite{GonGua}). One generalisation of this property is to consider a local Borsuk-Ulam problem in the sense of the following definition: a homotopy class $\beta \in [M,N]$ {\it has the Borsuk-Ulam property with respect to $\tau$} if for every representative $f : M \to N$ of $\beta$, there exists a point $x \in M$ such that $f(\tau(x))=f(x)$.

In~\cite{GonGuaLaa}, the Borsuk-Ulam problem for homotopy classes of maps between compact surfaces without boundary was studied, and the sets $ [\torus, \torus]$ and $[\klein, \klein]$ whose elements possess the Borsuk-Ulam property were characterised. By~\cite[Theorem~12]{GonGua}, for any involution $\tau: \torus \to \torus$, the triple $(\torus,\tau ; \klein)$ does not have the Borsuk-Ulam property. Using this information, in this paper we classify the homotopy classes of maps from $\torus$ to $\klein$ that have the Borsuk-Ulam property for the orientation-preserving free involution  $\tau_{1}$ of $\torus$. Our approach, which we now describe, is similar to that used in~\cite{GonGuaLaa}. First, as in~\cite[Theorems~12 and~19]{GonGuaLaa}, we identify $\pi_1(\torus,* )$ and $\pi_1(\klein,*)$ with the free Abelian group $\zsz$ and the (non-trivial) semi-direct product~$\zsdz$ respectively. These identifications will be helpful in formulating the results and in making explicit computations.

To prove our results, we will make use of the following algebraic description given in~\cite[Corollary~2.1]{GonKel} of the set $[\torus,\klein]$ in terms of pointed homotopy classes and the corresponding fundamental groups.

\begin{proposition}\label{prop:set_homotopy}
The set $[\torus,\klein]$ is in bijection with the subset of \linebreak $\hom(\zsz,\zsdz)$ whose elements are described as follows:	
\begin{multicols}{2}

Type 1: ${\allowdisplaybreaks
\begin{cases}
 (1,0) \mapsto (i,2s_1+1)\\
 (0,1) \mapsto (0,2 s_2 )
\end{cases}}$

Type 2: ${\allowdisplaybreaks
\begin{cases}
 (1,0) \mapsto (i,2s_1+1)\\
 (0,1) \mapsto (i,2 s_2+1)
\end{cases}}$

Type 3: ${\allowdisplaybreaks
\begin{cases}
 (1,0) \mapsto (0,2s_1)\\
 (0,1) \mapsto (i,2 s_2 +1)
\end{cases}}$

Type 4: ${\allowdisplaybreaks
\begin{cases}
 (1,0) \mapsto (r_1,2s_1)\\
 (0,1) \mapsto (r_2,2 s_2 ),
\end{cases}}$
\end{multicols}
\noindent where $i \in \{0,1 \}$ and $s_1, s_2 \in \z$ for Types~1,2 and~3, and $r_1, r_2, s_1, s_2 \in \z$ and $r_1 \geq 0$ for Type~4.
\end{proposition}

\begin{remark}\label{rem:homotopy_pi1}
The bijection of Proposition~\ref{prop:set_homotopy} may be obtained using standard arguments in homotopy theory that are described in detail in~\cite[Chapter V, Corollary 4.4]{White}, and more briefly in~\cite[Theorem~4]{GonGuaLaa}. In our specific case, the bijection is defined as follows: given a homotopy class $\beta \in [ \torus,\klein ]$, there exists a pointed map $f\colon\thinspace (\torus,*)\to(\klein,*)$ that gives rise to a representative of $\beta$ if we omit the basepoints. The induced homomorphism $f_\#\colon\thinspace \pi_1(\torus,*) \to \pi_1(\klein,*)$ is conjugate to exactly one of the elements of $\hom(\zsz,\zsdz)$ described in Proposition~\ref{prop:set_homotopy}, which we denote by $\beta_\#$. Note that $\beta_\#$ does not depend on the choice of $f$.
\end{remark}

In order to solve the Borsuk-Ulam problem for homotopy classes, we now des\-cribe the relevant involution of $\torus$. Consider the following short exact sequence:

\begin{align}
& 1 \to \pi_1(\torus)= \zsz \stackrel{i_1}{\longrightarrow}  \pi_1(\torus)= \zsz \stackrel{\theta_1}{\longrightarrow} \ztwo \to 1\label{eq:homo_tau_1}
\end{align}
where:
\begin{equation*}
	i_1: \begin{cases}
	 (1,0) \longmapsto (2,0)\\
	 (0,1) \longmapsto (0,1)
	\end{cases}
\theta_1: \begin{cases}
	 (1,0) \longmapsto \overline{1}\\
	 (0,1) \longmapsto \overline{0}.
	\end{cases}	
\end{equation*}
By standard results in covering space theory, there exists a double covering $c_1: \torus \to \torus$ whose induced homomorphism on the level of fundamental groups is $i_1$. If $\tau_1: \torus \to \torus$ is the non-trivial deck transformation associated with $c_1$, then $\tau_1$ is a free involution. Further, $\tau_{1}$ lifts to a homeomorphism $\widehat{\tau}_{1}: \rtwo \to \rtwo$, where $\widehat{\tau}_{1}(x,y)=(x+\frac{1}{2},y)$ for all $(x,y)\in \rtwo$. In this paper, we classify the elements of the set $[ \torus,\klein ]$ that possess the Borsuk-Ulam property with respect to $\tau_1$. This is achieved in the following theorem, which is the main result of this paper.

\begin{theorem}\label{th:BORSUK_TAU_1}
Given a non-zero integer $t$, let $e(t)$ be its $2$-adic evaluation. With the notation of Proposition~\ref{prop:set_homotopy}, let $\beta \in [ \torus,\klein ]$ and $\beta_\# \in \hom(\zsz,\zsdz )$. Then $\beta$ has the Borsuk-Ulam property with respect to $\tau_1$ if and only if one of the following conditions is satisfied:
\begin{enumerate}[(a)]
\item $\beta_\#$ is a homomorphism of Type~3.
\item $\beta_\#$ is a homomorphism of Type~4, where $s_1$ is odd and $r_2 \neq 0$, and additio\-nally $e(r_1) > e(r_2)$ if $r_{1}\neq 0$.
\end{enumerate}
\end{theorem}

Besides the introduction and an Appendix, this paper consists of three sections. In Section~\ref{sec:set_homo}, we show how to reduce the number of homotopy classes to be studied with respect to the Borsuk-Ulam property. In Section~\ref{sec:closure_sigma2}, we study the normal closure of $\sigma^2$ in $P_2(\klein)$, which is a free group of infinite rank. A convenient basis for this subgroup is obtained in the Appendix. In Section~\ref{sec:borsuk_1}, we prove Theorem \ref{th:BORSUK_TAU_1}.

The study of the free involution $\tau_{2}$ of $\torus$ for which the associated orbit space is the Klein bottle is the subject of work in progress.

%%%%%%%%%%%%%%%%%%%%%%%%%%%%%%%%%%%%
% Preliminary results
%%%%%%%%%%%%%%%%%%%%%%%%%%%%%%%%%%%%

\section{Some preliminary results}\label{sec:set_homo}

The following results will enable us to reduce the number of cases to be analysed in the proof of Theorems~\ref{th:BORSUK_TAU_1}.

\begin{lemma}\label{lem:reduction} 
Let $M$ and $N$ be topological spaces, let $\tau\colon\thinspace M \to M$ be a free involution, and let $H\colon\thinspace N \to N$ be a homeomorphism. Then the map $\mathcal{H}\colon\thinspace [M,N] \to [M,N]$ defined by $\mathcal{H}([f]) = [H \circ f]$ for all maps $f\colon\thinspace M \to N$ is a bijection. Further, if $\beta\in [M,N]$ is a homotopy class, then $\beta$ has the Borsuk-Ulam property with respect to $\tau$ if and only if $\mathcal{H}(\beta)$ has the Borsuk-Ulam property with respect to $\tau$.
\end{lemma}

\begin{proof}
Clearly the map $\mathcal{H}$ is a bijection whose inverse is given by $\mathcal{H}^{-1}([g]) = [H^{-1} \circ g]$. To prove the second part of the statement, let $\beta \in [M,N]$ be a homotopy class that has the Borsuk-Ulam property with respect to $\tau$, and let $g\in \mathcal{H}(\beta)$. Thus $H^{-1}\circ g\in \beta$, and hence there exists $x\in M$ such that $H^{-1}\circ g(x)=H^{-1}\circ g(\tau(x))$. Therefore $g(x)=g(\tau(x))$, and we conclude that $\mathcal{H}(\beta)$ has the Borsuk-Ulam property. The converse follows in a similar manner using $H^{-1}$.
\end{proof}

\begin{proposition}\label{prop:reduced_cases}
Let $\tau : \torus \to \torus$ be a free involution, and let $\beta, \beta' \in [\torus  ; \klein ]$ such that $\beta_\# , \beta'_\#$ are both of Type~1,~2 or~3. Suppose that the second coordinates of  $\beta_\#(\omega)$ and $\beta'_\#(\omega)$ are equal for all $\omega \in \pi_1 ( \torus , *)$ and the integer $i$ that defines the homomorphism $\beta_\#$ (resp.~$\beta'_\#$) is equal to $0$ (resp.\ $1$). Then $\beta$ has the Borsuk-Ulam property with respect to $\tau$ if and only if $\beta'$ does.
\end{proposition}

\begin{proof}
Let $h: \zsdz \to \zsdz$ be the homomorphism defined on the generators of $\zsdz$ by $h(1,0)=(1,0)$ and $h(0,1)=(1,1)$. Then $h$ is well defined, and  it is an isomorphism whose inverse $h^{-1}: \zsdz \to \zsdz$ is given by $h^{-1}(1,0)=(1,0)$ and  $h^{-1}(0,1)=(-1,1)$. By~\cite[Theorem~5.6.2]{Zies}, there exists a homeomorphism $H: (\klein,*)\to (\klein,*)$ such that $H_\#=h$ and $H^{-1}_\#=h^{-1}$. Suppose that $\beta_\#$ and $\beta'_\#$ are both of Type~1, and that they satisfy the hypothesis of the statement, and let $f\colon\thinspace (\torus,*) \to (\klein,*)$ be a representative map of $\beta$. Without loss of generality, we may suppose that $f_\# = \beta_\#$. Assume that $\beta$ has the Borsuk-Ulam property with respect to $\tau$. Then:
\begin{align*}
(H \circ f)_\#(1,0)&=h(\beta_\# (1,0))= h(0,2s_1+1)= h(0,1)^{2s_1+1}\\
&=((1,1)(1,1))^{s_1}(1,1)=(0,2)^{s_1}(1,1)=(1, 2s_1+1), \text{and}\\
(H \circ f)_\#(0,1)&=h(\beta_\# (0,1))= h(0,2 s_2)=h(0,1)^{2 s_2}=((1,1)(1,1))^{s_2} =(0,2)^{s_2}=(0, 2 s_2).
\end{align*}
Then $H \circ f$ is a representative of $\beta'$, and the conclusion follows from Lemma~\ref{lem:reduction}. The converse follows in a similar way using $H^{-1}$ instead $H$. The arguments for the cases of homomorphisms of Types~2 and~3 are analogous, and the details are left to the reader.
\end{proof}

\begin{remark}
Using Lemma~\ref{lem:reduction}, one may show that Proposition~\ref{prop:reduced_cases} holds in more generality. However the above statement will be sufficient for our purposes.
\end{remark}

%%%%%%%%%%%%%%%%%%%%%%%%%%%%%%%%%%%%
% The normal closure of \sigma^2
%%%%%%%%%%%%%%%%%%%%%%%%%%%%%%%%%%%%

\section{The normal closure of $\sigma^2$ in $P_2(\klein)$}\label{sec:closure_sigma2}

Let $S$ be a compact, connected surface without boundary. The ordered $2$-point configuration space of $S$ is defined by $F_2(S)=\{(x, y)\in S \times S \ | \ x \neq y \}$, $D_2(S)$ is the orbit space of $F_2(S)$ by the free involution $\tau_S : F_2(S) \to F_2(S)$ defined by $\tau_S(x,y)=(y,x)$, and $P_2(S)=\pi_1(F_2(S))$ and $B_2(S)= \pi_1(D_2(S))$ are the pure and full $2$-string braid groups  of $S$ respectively~\cite{FadNeu}. We have a short exact sequence:
\begin{equation}\label{eq:sesbraid}
1 \to P_{2}(S) \to B_{2}(S) \stackrel{\pi}{\to} \ztwo \to 1,
\end{equation}
where $\pi : B_2(S)\to \ztwo$ is the homomorphism that to an element of $B_{2}(S)$ associates its permutation. If $p_1: F_2(S)\to S$ is the projection onto the first coordinate, the map $(p_1)_\# : P_2(S) \to \pi_1(S)$ may be interpreted geometrically as the surjective homomorphism that forgets the second string. Let $\sigma \in B_2(S) \setminus P_2(S)$ be the Artin generator of $B_2(S)$ that geometrically swaps the two basepoints. Then $\sigma^2 \in P_2(S)$, and the normal closure of $\sigma^2$ in $P_{2}(S)$, which we denote by $\closure$, is also the normal closure of $\sigma^2$ in $B_{2}(S)$ by~(\ref{eq:sesbraid}). In this section, we shall take $S$ to be the Klein bottle, and we will show that $\closure$ is a free group of countably-infinite rank for which we shall exhibit a basis. We will also express certain elements of $\closure$ in this basis.

The following proposition summarises some results of~\cite[Section 4]{GonGuaLaa} regarding the structure of $P_2(\klein)$ and the action by conjugation by $\sigma$ on this group.

\begin{proposition}{\cite[Theorems~19 and~20]{GonGuaLaa}}\label{prop:presentation_p2}
The group $P_2(\klein)$ is isomorphic to the semi-direct product $F(u,v) \rtimes_\theta (\zsdz)$, where $F(u,v)$ is the free group of rank $2$ on the set $\{u,v\}$ and the action $\theta: \zsdz \to \operatorname{Aut}(F(u,v))$ is defined as follows:
\begin{equation*}
\theta(m,n): \begin{cases}
u \mapsto B^{m-\delta_n} u^{\varepsilon_n} B^{-m+\delta_n}\\
v \mapsto B^m v u^{-2m} B^{-m+\delta_n}\\
B  \mapsto B^{\varepsilon_n},
\end{cases} 
\end{equation*}
where $\delta_n=\begin{cases}
0 & \text{if $n$ is even}\\
1 & \text{if $n$ is odd,}
\end{cases}$ $\varepsilon_n=(-1)^n$ and $B=uvuv^{-1}$. With respect to this description, the following properties hold:
\begin{itemize}
\item the element $\sigma \in B_2(\klein)$ satisfies $\sigma^2 =(B ; 0,0)$.
	
\item if $l_\sigma: P_2(\klein) \to P_2(\klein)$ is the homomorphism defined by $l_\sigma(b)=\sigma b \sigma^{-1}$ for all $b \in P_2 (\klein)$, then:
\begin{align*}
l_\sigma(u^r;0,0)&=((Bu^{-1})^r B^{-r} ; r,0) & l_\sigma(\id;m,0)&=(\id ; m,0)\\
l_\sigma(v^s;0,0)&=((uv)^{-s} (u B)^{\delta_s} ; 0,s) & l_\sigma(\id;0,n) &=(B^{\delta_n} ; 0,n)\\
l_\sigma(B;0,0)&=(B ; 0,0) &&
\end{align*}
for all $m,n,r,s \in \z$, where the symbol $\id$ denotes the trivial element of $F(u,v)$.

\item the homomorphism $(p_1)_\#: P_2(\klein) \to \pi_1(\klein)=\zsdz$ satisfies  $(p_1)_\# (w  ; r,s)=(r,s)$.
\end{itemize}
\end{proposition}

From now on, we identify $P_2(\klein)$ with $F(u,v) \rtimes_\theta (\zsdz)$ without further comment.
		
\begin{remark}\label{rem:theta_parity}
It follows from Proposition~\ref{prop:presentation_p2} that for all $m,n\in \z$, the automorphism $\theta(m,n): F(u,v)\to F(u,v)$ depends only on $m$ and the parity of $n$, in particular $\theta(m,n)=\theta(m,\delta_n)$.
\end{remark}		

Consider the following maps:
\begin{equation*}
\text{$\iota: \left\{ \begin{aligned}
F(u,v) & \longrightarrow P_2(\klein)\\
w & \longmapsto (w ; 0,0)
\end{aligned}\right.$ and
$p_F: \left\{ \begin{aligned}
P_2 (\klein) & \longrightarrow  F(u,v)\\
(w ; m,n)& \longmapsto w.
\end{aligned}\right.$}
\end{equation*}
Note that $\iota$ is a homomorphism, but due to the action $\theta$, $p_F$ is not. Consider the map $\rho: F(u,v) \to F(u,v)$ defined by:
		\begin{equation}\label{eq:defi_rho}
\rho=p_F \circ l_\sigma \circ \iota,	
		\end{equation}
and the homomorphism $g: F(u,v) \to \zsdz$ defined on the basis $\{u,v\}$ by:
\begin{equation}\label{defi_g}
\begin{cases}
g(u)=(1,0)\\
g(v)=(0,1).
\end{cases}
\end{equation}
Using Proposition~\ref{prop:presentation_p2} and~(\ref{eq:defi_rho}), we obtain the following commutative diagram:
\begin{equation}\label{eq:diag_rho_g}\begin{gathered}\xymatrix{
& & &  F(u,v)\\
	F(u,v) \ar[rr]^-\iota \ar[rrrd]_-{g} \ar[rrru]^-{\rho} 
	& &  P_2(\klein)\ar[r]^-{ l_\sigma }
& P_2(\klein)=F(u,v) \rtimes_\theta(\zsdz)\ar[d]^-{ (p_1)_\#}\ar[u]_-{p_F}\\
	& & &  \zsdz=\pi_1(\klein),
	}\end{gathered}\end{equation}
from which it follows that:
\begin{equation}\label{eq:escrita_l_sigma_rho_g}
l_\sigma(w ; 0,0) =(\rho(w) ; g(w))
\end{equation}
for all $w \in F(u,v)$. Further if $w,z \in F(u,v)$ then:
		\begin{align*}
\rho(wz) & = (p_F \circ l_\sigma)(w z ; 0,0)=p_F (l_\sigma(w; 0,0) \ldotp l_\sigma(z ; 0,0)) =p_F \bigl((\rho(w) ; g(w)) \ldotp (\rho(z) ; g(z))\bigr)\\
&  =p_F \bigl(\rho(w) \theta(g(w))(\rho(z)) ; g(w)g(z)\bigr)  =\rho (w) \theta(g(w))(\rho(z)).
		\end{align*}
Thus the map $\rho: F(u,v) \to F(u,v)$ is not a homomorphism, but if $w\in \ker g$ then:
\begin{equation}\label{theta_kerg}
\rho(wz)=\rho(w) \rho(z).
\end{equation}
In Theorem~\ref{th:basis_B} of the Appendix, we prove that:
		\begin{equation}\label{base_kerg}
\ker g=\left\langle B_{k,l} := v^k u^l B u^{-l} v^{-k},\, k,l \in \z \ |-\right\rangle.
		\end{equation}
Let $w \in F(u,v)$ and $(m,n) \in \zsdz$. Using Proposition~\ref{prop:presentation_p2}, we see that $B\in \ker{g}$, and then that:
\begin{align}\label{eq:theta_kerg}
\theta(m,n)(w B w^{-1}) = \theta(m,n)(w) B^{\varepsilon_n} \theta(m,n)(w)^{-1} \in \ker{g}.
\end{align}
Further,
\begin{align}\label{eq:l_sigma_kerg}
l_\sigma(w B w^{-1} ; 0 ,0) &= l_\sigma(w ; 0,0) l_\sigma(B ; 0,0) l_\sigma(w ; 0,0)^{-1}  \stackrel{(\ref{eq:diag_rho_g})}{=} (\rho(w) ; g(w)) (B ; 0,0)(\rho(w) ; g(w))^{-1} \notag\\
&=(\rho(w) \theta(g(w))(B); g(w))(\theta(g(w)^{-1}) (\rho(w)^{-1}) ; g(w)^{-1})  \notag\\
&=(\rho(w) \theta(g(w))(B) \rho(w)^{-1} ; 0,0). 
\end{align}
Composing~(\ref{eq:l_sigma_kerg}) by $p_F$, it follows from~(\ref{eq:diag_rho_g}) and~(\ref{eq:theta_kerg}) that:
\begin{equation}\label{rho_kerg}
\rho(w B w^{-1})= \rho(w) \theta(g(w))(B) \rho(w)^{-1} \in \ker g .
\end{equation}
In particular, $\rho(B_{k,l})\in \ker{g}$ for all $k,l\in \z$, and using~(\ref{theta_kerg}) and~(\ref{base_kerg}), the restriction of $\rho$ to $\ker{g}$ yields an endomorphism of $\ker{g}$, which we also denote by $\rho$. Further, $\theta(m,n)(B_{k,l})\in \ker{g}$ for all $k,l,m,n\in \z$ using~(\ref{eq:theta_kerg}), and thus $\theta$ determines a homomorphism from $\zsdz$ to $\operatorname{Aut}(\ker{g})$, which we also denote by $\theta$. Note that for all $m,m\in \z$, the endomorphism $\theta(m,n): \ker{g} \to \ker{g}$ is indeed surjective. To see this, let $k,l\in \z$, and let $\xi\in F(u,v)$ be such that $\theta(m,n)(\xi)= v^{k}u^{l}$. Then $\theta(m,n)(\xi B^{\varepsilon_{n}} \xi^{-1})=B_{k,l}$, and $\xi B^{\varepsilon_{n}} \xi^{-1} \in \ker{g}$ because $B\in \ker{g}$. Hence the image of $\theta(m,n): \ker{g} \to \ker{g}$ contains the basis $\{ B_{k,l}\}_{k,l\in \z}$ of $\ker{g}$, and so this homomorphism is surjective.

The following result describes the subgroup $\gsigma$.

\begin{proposition}\label{prop:prop_gsigma} 
The injective homomorphism $\iota: F(u,v)\to P_{2}(\klein)$ defined by $\iota(w) = ( w ; 0,0)$ for all $w\in F(u,v)$, restricts to an isomorphism between $\ker{g}$ and $\gsigma$, which we also denote by $\iota$. In particular, $\gsigma$ is a free group of infinite rank for which a basis is given by $\{(B_{k,l} ; 0,0)\}_{k,l\in \z}$, and $l_{\sigma}(\gsigma)\subset \ker{(p_{1})_{\#}}$. Up to this isomorphism, the homomorphisms $\theta: \zsdz \to \operatorname{Aut}(\ker{g})$ and $\rho: \ker{g} \to \ker{g}$ induce homomorphisms $\zsdz \to \operatorname{Aut}(\gsigma)$ and $\gsigma \to \gsigma$, which we also denote by $\theta$ and $\rho$ respectively. Further, the following diagram is commutative:
$$\xymatrix{
\gsigma \ar[r]^-\rho \ar@{^{(}->}[d] &  \gsigma \ar@{^{(}->}[d]\\
P_2(\klein) \ar[r]^-{l_\sigma} & P_2(\klein) .
}$$
\end{proposition}

\begin{proof}
Let $H=\iota(\ker{g})\subset P_2(\klein)$. We start by showing that $\gsigma=H$. To see this, first note that $\sigma^2 =(B ; 0,0 )$ and $\iota(B_{k,l})\in \gsigma$ for all $k,l\in \z$, hence $H\subset \gsigma$ by~(\ref{base_kerg}). Conversely, for all $w \in F(u,v)$, $q \in \zsdz$ and $k,l\in \z$, we have:
\begin{align}\label{conjugacao_in_h}
(w ; q)(B_{k,l} ; 0,0 )(w ; q)^{-1} = & (w \theta(q)(B_{k,l}) ; q)(\theta(q^{-1})(w^{-1}) ; q^{-1})  = (w \theta(q)(B_{k,l})w^{-1} ; 0,0) .
\end{align}
Since $\theta(q)(B_{k,l})\in \ker g$ and $\ker{g}$ is normal in $ F(u,v)$, it follows that $w \theta(q)(B_{k,l} )w^{-1} \in \ker g$, and so $(w ; q)(B_{k,l} ; 0,0 )(w ; q)^{-1}  \in \iota(\ker g)=H$ by~(\ref{conjugacao_in_h}). Hence $H$ is a normal subgroup of $P_2(\klein)$ by~(\ref{base_kerg}), and since $\sigma^{2}\in H$, we conclude that $\gsigma \subset H$. Thus $\gsigma=H$ as required. Since $\iota$ is injective, the induced homomorphism $\iota: \ker{g} \to \gsigma$ is an isomorphism, and the image of the basis $\{ B_{k,l}\}_{k,l\in \z}$ by $\iota$ yields a basis of the free group $\gsigma$. The commutativity of the given diagram follows by considering the images of the elements of this basis $\{ (B_{k,l}; 0,0)\}_{k,l\in \z}$ and using~(\ref{eq:l_sigma_kerg}). 
\end{proof}

\begin{remark}\label{rem:kerg_sigma}
Although $\ker{g}$ and $\gsigma$ are isomorphic by Proposition~\ref{prop:prop_gsigma}, our results will be stated in terms of $\gsigma$, since we will formulate most of our equations in this group.
\end{remark}

The following result provides a normal form for elements of $F(u,v)$ in terms of $g$ and $\gsigma$. 

\begin{proposition}\label{prop:normal_form}
Let $w \in F(u,v)$, and let $g(w)=(r,s)$. Then there exists a unique element $x \in \gsigma$ such that $w=u^r v^s x$.
\end{proposition}
		
\begin{proof}
Let $w$ be as in the statement, and let $x= v^{-s} u^{-r} w$. Then $w=u^r v^s x$, and:
\begin{align*}
g(x)&=g(v^{-s} u^{-r} w)=  g(v)^{-s} g(u)^{-r} g(w)=(0,-s)(-r,0) (r,s)=(0,-s)(0,s)=(0,0).
\end{align*}
So $x\in \gsigma$. Clearly $x$ is unique.
\end{proof}

Let $p,q \in \z$. Since $\gsigma$ is a normal subgroup of $F(u,v)$, the following homomorphism is well defined:
\begin{equation}\label{eq:defi_conjugacao}
\begin{array}{rrcl}
c_{p,q}: & \gsigma & \longrightarrow & \gsigma\\
& x & \longmapsto & v^p u^q x u^{-q} v^{-p}.
\end{array}
\end{equation}

\begin{lemma}\label{lem:homo_gsigma}
Let $k,l,p,q\in \z$, and let $(m,n) \in \zsdz$. With the notation of Proposition~\ref{prop:presentation_p2}, there exist $\gamma, \lambda, \eta \in \gsigma$ such that:		
\begin{enumerate}[(a)]	
\item\label{it:homo_gsigmaa} $\theta(m,n)(B_{k,l})= \gamma B_{k,\varepsilon_n l-2 \delta_k m}^{ \varepsilon_n}\gamma^{-1}$.
\item\label{it:homo_gsigmab} $\rho(B_{k,l})= \lambda B_{-k, \varepsilon_{(k+1)} l }^{\varepsilon_k} \lambda^{-1}$.
\item\label{it:homo_gsigmac} $c_{p,q}(B_{k,l})= \eta B_{k+p, l+\varepsilon_{k} q} \eta^{-1}$.
\end{enumerate}
\end{lemma}

\begin{proof}
During the proof, we will make use freely of Proposition~\ref{prop:presentation_p2}. First, we have:
\begin{align*}
\theta(m,n)(B_{k,l}) & = \theta(m,n)(v^k u^l B u^{-l} v^{-k}) = \theta(m,n)(v^k u^l)B^{\varepsilon_n} \theta(m,n)(v^k u^l )^{-1}  = \gamma B_{k, \varepsilon_n l-2 \delta_k m }^{\varepsilon_n} \gamma^{-1},
\end{align*}
where $\gamma=\theta(m,n)(v^k u^l) u^{ \varepsilon_{n+1} l+2 \delta_k m} v^{-k} \in F(u,v)$. To complete the proof of item~(\ref{it:homo_gsigmaa}), it remains to show that $\gamma\in \ker{g}$. Since:
\begin{align*}
\gamma %&=\theta(m,n)(v^k u^l) u^{ \varepsilon_{(n+1)} l+2 \delta_k m} v^{-k}\\
& =(B^m v u^{-2m} B^{-m+\delta_n })^k(B^{m-\delta_n} u^{ \varepsilon_n}B^{-m+\delta_n})^l u^{ \varepsilon_{n+1} l+2 \delta_k m} v^{-k},
\end{align*}
and $B\in \ker{g}$, it follows that:
\begin{align*}
g(\gamma)&=\left((0,1)(-2m,0) \right)^k(\varepsilon_n l,0)(\varepsilon_{n+1} l+2 \delta_k m,0)(0,-k)=(2 m,1 )^k(\varepsilon_n l+\varepsilon_{n+1} l+2 \delta_k m ,-k)\\
& =(2 \delta_k m,k)(2 \delta_k m ,-k) =(2 \delta_k m+2 \delta_k \varepsilon_k m,0)=(0,0),
\end{align*}
using the fact that $\varepsilon_k=-1$ if $k$ is odd, and $\delta_{k}=0$ if $k$ is even. Hence $\gamma\in \ker{g}$.
To prove item~(\ref{it:homo_gsigmab}), first note that:
\begin{equation}\label{eq:gvuB}
\theta(g(v^k u^l))(B) =\theta((0,k)(l,0))(B)=\theta(\varepsilon_k l,k)(B)= B^{\varepsilon_k}. 
\end{equation}
By~(\ref{rho_kerg}) and~(\ref{eq:gvuB}), we have:
\begin{align*}
\rho(B_{k,l}) &= \rho(v^k u^l B u^{-l} v^{-k})=\rho(v^k u^l)\theta(g(v^k u^l))(B) \rho(v^k u^l)^{-1} 
 = \lambda B_{-k, \varepsilon_{(k+1)} l }^{\varepsilon_k} \lambda^{-1},
\end{align*}
where $\lambda=\rho(v^k u^l) u^{\varepsilon_k  l} v^k \in F(u,v)$. It remains to show that $\lambda\in \ker{g}$. We have:
\begin{align*}
g(\lambda) &= g(\rho(v^k u^l) u^{\varepsilon_k  l} v^k)= g((p_F \circ l_\sigma \circ i)(v^k u^l)) \ldotp (\varepsilon_k  l,k)\\
&= g(p_{F} (((uv)^{-k} (uB)^{\delta_k} ; 0,k)((B u^{-1})^l B^{-l} ; l,0))) \ldotp (\varepsilon_k  l,k)\\
& = g((uv)^{-k} (uB)^{\delta_k} \theta(0,k)((B u^{-1})^l B^{-l})) \ldotp (\varepsilon_k  l,k)\\
&= (1,1)^{-k} (\delta_k,0) \ldotp g((B^{ \varepsilon_k}(B^{- \delta_k} u^{ \varepsilon_k} B^{ \delta_k})^{-1} )^l B^{-\varepsilon_{k} l}) \ldotp (\varepsilon_k  l,k)\\
&= (\delta_{-k}, -k) (\delta_k,0) (-\varepsilon_k l,0) (\varepsilon_k  l,k)=(\delta_{-k}+\varepsilon_{-k} \delta_{k}, 0)=(0,0),
\end{align*}
since $\delta_{k}=0$ if $k$ is even, and $\varepsilon_k=-1$ if $k$ is odd. Hence $\lambda\in \ker{g}$. Finally we prove  item~(\ref{it:homo_gsigmac}). We have:
\begin{align*}
c_{p,q}(B_{k,l}) &=  v^p u^q B_{k,l} u^{-q} v^{-p}=v^p u^q v^k u^l B u^{-l} v^{-k} u^{-q} v^{-p} = \eta B_{k+p,l +\varepsilon_k q} \eta^{-1},
\end{align*}
where $\eta=v^p u^q v^k u^{\varepsilon_{k+1} q}v^{-k-p} \in F(u,v)$. It remains to show that $\eta \in \ker{g}$. This is indeed the case because:
\begin{align*}
g(\eta) & = (0,p)(q,k)(\varepsilon_{k+1} q,-k-p)=(0,p)(q+\varepsilon_k \varepsilon_{k+1} q,-p) =(0,p)(0,-p)  =(0,0),
\end{align*} 
which completes the proof.
\end{proof}

Let $\gsigmab$ denote the Abelianisation of the group $\gsigma$. By abuse of notation, for all $k,l\in \z$, we denote the image of a generator $B_{k,l}$ in $\gsigmab$ by $B_{k,l}$. By Proposition~\ref{prop:prop_gsigma}, $\gsigmab$ is the free Abelian group for which $\{ B_{k,l} := v^k u^l B u^{-l} v^{-k}\; | \; k,l\in \z \}$ is a basis, namely:
\begin{equation*}
\gsigmab= \bigoplus_{k,l \in \z} \z \left[ B_{k,l} \right].
\end{equation*}
For all $(m,n) \in \zsdz$ and $p,q \in \z$, the endomorphisms $\theta(m,n), \rho$ and $(c_{p,q})$ of $\gsigma$ induce endomorphisms $\theta(m,n)_\ab, \rho_\ab$ and $(c_{p,q})_\ab$ of $\gsigmab$ respectively, and by Lemma~\ref{lem:homo_gsigma}, for all $k,l\in \z$, they satisfy:
%
%\begin{equation}\label{eq:homo_gsigmab}
%\text{$\theta(m,n)_\ab(B_{k,l})= { \varepsilon_n}B_{k ,\varepsilon_n l-2 \delta_k m}$, $\rho_\ab(B_{k.l})= {\varepsilon_k} B_{-k,\varepsilon_{(k+1)} l}$ and $(c_{p,q})_\ab(B_{k,l})= B_{k+p, l+\varepsilon_{k} q}$.}
%\end{equation}
%
\begin{align}
\theta(m,n)_\ab(B_{k,l}) & =  { \varepsilon_n}B_{k ,\varepsilon_n l-2 \delta_k m} \label{eq:homo_gsigmab_theta} \\
\rho_\ab(B_{k.l}) & = {\varepsilon_k} B_{-k,\varepsilon_{(k+1)} l}, \text{ and} \label{eq:homo_gsigmab_rho} \\
(c_{p,q})_\ab(B_{k,l}) & =  B_{k+p, l+\varepsilon_{k} q}. \label{eq:homo_gsigmab_cpq}
\end{align}

Let $k, l \in \z$ and $r \in \{ 0 , 1 \}$. If $x$ and $y$ are elements of a group, let $[x,y]=xyx^{-1}y^{-1}$ denote their commutator. In the rest of the paper, we will be particularly interested in the following elements of $F(u,v)$:
\begin{enumerate}[(I)]
\item $T_{k,r}= u^{k} (B^{\varepsilon_{r}} u^{-\varepsilon_{r}})^{k\varepsilon_{r}}$.
\item $I_{k}=v^{k} (v B )^{-k}$.
\item $O_{k,l}=\left[ v^{2k},u^l \right]$.
\item $J_{k,l}=v^{2k}(v u^l  )^{-2k}$.
\end{enumerate}
As we shall now see, $T_{k,r}$, $I_{k}$, $O_{k,l}$ and $J_{k,l}$ are elements of $\gsigma$, and their projections into $\gsigmab$ will be denoted by $\widetilde{T}_{k,r}$, $\widetilde{I}_k$, $\widetilde{O}_{k,l}$ and $\widetilde{J}_{k,l}$ respectively. The following result describes the decomposition of these Abelianisations in terms of the basis $\{ B_{k,l} \}_{k,l}$. If $l\in \z$, let $\sigma_{l}$  denote its sign, \emph{i.e}\ $\sigma_{l}=1$ if $l>0$, $\sigma_{l}=-1$ if $l<0$, and $\sigma_{l}=0$ if $l=0$.

\begin{proposition}\label{prop:words}
Let $k, l \in \z$  and $r \in \{ 0 , 1 \}$. 
\begin{enumerate}[(a)]
\item  The elements $T_{k,r}, I_{k}, O_{k,l}$ and $J_{k,l}$  belong to $\gsigma$.

\item\label{it:wordsb}  If $k=0$ then $\widetilde{T}_{0,r} = \widetilde{I}_0 = 0$, and if $k=0$ or $l=0$ then $\widetilde{O}_{k,l}=\widetilde{J}_{k,l}=0$. 

\item\label{it:wordsc} For all $k,l \neq 0$:
\begin{align*}
\widetilde{T}_{k,r}&= \sigma_{k} \sum_{i=1}^{\sigma_{k}k} B_{0,\sigma_{k}(i+(\sigma_{k}(1-2r)-1)/2)}\\
\widetilde{I}_k&=- \sigma_{k} \sum_{i=1}^{\sigma_k k} B_{\sigma_k  i + ( 1 - \sigma_k )/2 ,0}\\
\widetilde{J}_{k,l}&=-\sigma_{k}\sigma_{l} \sum_{i=1}^{\sigma_{k}k} \sum_{j=1}^{\sigma_{l}l} B_{\sigma_{k}(2i-1), \sigma_{l}(j-(1+\sigma_{l})/2)}\\
\widetilde{O}_{k,l}& = \sigma_k \sigma_l \sum_{i=1}^{\sigma_k k} \sum_{j=1}^{\sigma_l l} \bigl(
B_{\sigma_k ( 2i - 1 ), - \sigma_l j + ( \sigma_l -1 )/2} 
- B_{ \sigma_k (2i-1) -1 , \sigma_l j - (1 + \sigma_l)/2}
\bigr).
\end{align*}
\end{enumerate}
\end{proposition}

The proof of Proposition~\ref{prop:words}, which is divided into the following four lemmas, consists essentially in manipulating each of the elements to obtain recurrence relations, and using induction to obtain expressions for $T_{k,r}, I_{k}, O_{k,l}$ and $J_{k,l}$ in $\gsigma$. Part~(\ref{it:wordsc}) of Proposition~\ref{prop:words} is obtained by Abelianising these expressions.

\begin{lemma}\label{lem:word_t_w}
If $k \in \z$ and $r\in \{ 0,1\}$ then $\displaystyle T_{k,r}= \prod_{i=1}^{\sigma_{k}k} B_{0,k-\sigma_{k}i-r+(\sigma_{k}+1)/2}^{\sigma_{k}}$.
\end{lemma}

\begin{proof} 
Let $r\in \{ 0,1 \}$. Then clearly $T_{0,r}=\id$ and $T_{1,r}=u(B^{\varepsilon_{r}} u^{-\varepsilon_{r}})^{\varepsilon_{r}}= B_{0,1-r}$, so the result holds for $k=1$. Suppose that the formula is valid for some $k\geq 1$. Then:
\begin{align*}
T_{k+1,r} &=u^{k+1} (B^{\varepsilon_{r}} u^{-\varepsilon_{r}})^{(k+1)\varepsilon_{r}}=u\ldotp u^{k} (B^{\varepsilon_{r}} u^{-\varepsilon_{r}})^{k\varepsilon_{r}}\ldotp (B^{\varepsilon_{r}} u^{-\varepsilon_{r}})^{\varepsilon_{r}}\\
&=u T_{k,r}u^{-1}\ldotp u (B^{\varepsilon_{r}} u^{-\varepsilon_{r}})^{\varepsilon_{r}}
= u\left( \prod_{i=1}^{k} B_{0,k-i+1-r} \right)u^{-1}\ldotp B_{0,1-r}=\prod_{i=1}^{k+1} B_{0,k+1-i+1-r},
\end{align*}
and so the result holds for all $k\geq 0$ by induction. Suppose that $k\geq 1$. Then:
\begin{align*}
T_{-k,r}&= u^{-k} (B^{\varepsilon_{r}} u^{-\varepsilon_{r}})^{-k\varepsilon_{r}}=u^{-k} (u^{k}(B^{\varepsilon_{r}} u^{-\varepsilon_{r}})^{k\varepsilon_{r}})^{-1}u^{k}=u^{-k} T_{k,r}^{-1}u^{k}\\
&=\left( \prod_{i=1}^{k} B_{0,-i+1-r}\right)^{-1} =\prod_{i=1}^{k} B_{0,-k+i-r}^{-1}
\end{align*}
using the result for $T_{k,r}$ in the case $k\geq 1$, which proves the formula for all $k\in \z$.
\end{proof}

\begin{lemma}\label{lem:word_i}
Let $k \in \z$. If $k=0$ then $I_k=\id$, and if $k \neq 0$ then:
\begin{equation*}
\text{$I_k = \left( \displaystyle \prod_{i=1}^{\sigma_k k} B_{i + k ( 1 -\sigma_k) / 2 ,0}\right)^{-\sigma_k}$.}
\end{equation*}
\end{lemma}

\begin{proof}
If $k=0$, then clearly $I_0=\id$. If $k=1$, $I_1= v (vB)^{-1}=v B^{-1} v^{-1} = B_{1,0}^{-1}$, and result holds in this case. Suppose that the formula for $I_{k}$ holds for some $k\geq 1$, and let us prove the formula for $k+1$. We have:
\begin{align*}
I_{k+1} &=  v^{k+1} ( v B)^{-k-1} = v v^k (v B)^{-k} v^{-1} v (v B)^{-1} = v I_k v^{-1}	 I_1 \\
&= v \left( \prod_{i=1}^k B_{i,0} \right)^{-1} v^{-1} B_{1,0}^{-1} 
 = \left( \prod_{i=1}^k B_{i+1,0} \right)^{-1}  B_{1,0}^{-1} = \left( \prod_{i=1}^{k+1} B_{i,0} \right)^{-1},
\end{align*}
so by induction, the formula for $I_{k}$ holds for all $k\geq 0$. If $k <0 $ then $-k > 0$ and so:
\begin{align*}
I_k = & v^k (v B)^{-k}	 = v^k (v B)^{-k} v^k v^{-k} = v^k \left( v^{-k} (vB)^k \right)^{-1} v^{-k} = v^k I_{-k}^{-1} v^{-k} \\
= & v^k \left( \prod_{i=1}^{-k} B_{i,0} \right) v^{-k} = \prod_{i=1}^{-k} B_{k+i,0}.
\end{align*}
It follows that the formula given in the statement holds for all $k\neq 0$. 
\end{proof}

\begin{lemma}\label{lem:word_j}
Let $k,l \in \z$, let $\varepsilon \in \{-1,1 \}$, and let $\omega= \begin{cases}
0 & \text{if $\varepsilon=-1$}\\
1 & \text{if $\varepsilon=1$.} 
\end{cases}$
%, let $J_{k,l}= v^{2k} (vu^l)^{-2k} \in F(u,v)$. 
If $k=0$ or $l=0$, then $J_{k,l}=\id$, and if $k,l>0$, then:
\begin{equation*}%\label{eq:jkel}
\text{$J_{k,\varepsilon l}=\displaystyle \prod_{i=1}^k \left(\prod_{j=1}^l B_{2k-2i+1,\varepsilon j-\omega}^{-\varepsilon} \right)$ and $J_{-k,\varepsilon l}=\displaystyle \left(\prod_{i=1}^k \left(\prod_{j=1}^l B_{-2i+1,\varepsilon j-\omega}^{-\varepsilon} \right) \right)^{-1}$.}
\end{equation*}
\end{lemma}

\begin{proof}
If $k=0$ or $l=0$, then clearly $J_{k,l}=\id$. Now let $k=1$ and $\varepsilon\in \{1,-1\}$. Then for all $l\geq 0$, we have: 
\begin{align}
J_{1,\varepsilon(l+1)} &= v^2(vu^{\varepsilon(l+1)})^{-2}=v^2u^{-\varepsilon}u^{-\varepsilon l}v^{-1} u^{-\varepsilon}u^{-\varepsilon l}v^{-1}\notag\\
&= v^2(u^{-\varepsilon l}v^{-1}u^{-\varepsilon l}v^{-1})vu^{\varepsilon l}v u^{-\varepsilon}v^{-1} u^{-\varepsilon } u^{-\varepsilon l}v^{-1}\notag\\
&= J_{1,\varepsilon l}vu^{\varepsilon l+(\varepsilon-1)/2}(u^{(1-\varepsilon)/2}vu^{-\varepsilon}v^{-1}u^{-\varepsilon+(\varepsilon-1)/2})u^{-\varepsilon l+(1-\varepsilon)/2}v^{-1}\notag\\
&=J_{1,\varepsilon l}vu^{\varepsilon l+(\varepsilon-1)/2} B^{-\varepsilon} u^{-\varepsilon l+(1-\varepsilon)/2}v^{-1}= J_{1,\varepsilon l} B_{1, \varepsilon l+(\varepsilon-1)/2}^{-\varepsilon}.\label{eq:J1eps}
\end{align}
Hence $J_{1,1}=B_{1,0}^{-1}$ and $J_{1,-1}=B_{1,-1}$, which are in agreement with the expressions of the statement of the lemma for $k=l=1$. Suppose now that these expressions hold for $k=1$ and some $l\geq 1$. Now $\varepsilon l+(\varepsilon-1)/2=\varepsilon (l+1)-(\varepsilon+1)/2=\varepsilon (l+1)-\omega$, so by~(\ref{eq:J1eps}), we have:
\begin{equation*}
J_{1,\varepsilon(l+1)} =J_{k,\varepsilon l}= \left( \prod_{j=1}^l B_{1,\varepsilon j-\omega}^{-\varepsilon} \right) B_{1, \varepsilon l+(\varepsilon-1)/2}^{-\varepsilon}=\prod_{j=1}^{l+1} B_{1,\varepsilon j-\omega}^{-\varepsilon},
\end{equation*}
and thus the expressions of the statement hold for all $k,l\geq 1$ by induction. From this, for all $k,l>0$ and $\varepsilon\in \{ 1,-1\}$, it follows that:
\begin{align*}
J_{-k,\varepsilon l}&= v^{-2k}(v u^{\varepsilon l} )^{2k}=v^{-2k} \left(v^{2k}(v u^{\varepsilon l} )^{-2k} \right)^{-1} v^{2k} = v^{-2k} J_{k,\varepsilon l}^{-1} v^{2k}\\
&=  v^{-2k} \left(\prod_{i=1}^k \left(\prod_{j=1}^l B_{2k-2i+1,\varepsilon j-\omega}^{-\varepsilon} \right) \right)^{-1} v^{2k} = \left(\prod_{i=1}^k \left(\prod_{j=1}^l B_{-2i+1,\varepsilon j-\omega}^{-\varepsilon} \right) \right)^{-1},
\end{align*}
which completes the proof of the lemma.
\end{proof}

\begin{lemma}\label{lem:word_o}
Let $k,l \in \z$, let $\varepsilon \in \{-1,1 \}$, and let $\omega= \begin{cases}
0 & \text{if $\varepsilon=-1$}\\
1 & \text{if $\varepsilon=1$.} 
\end{cases}$
If $k=0$ or $l=0$, then $O_{k,l}=\id$. If $k,l>0$, then for all $1 \leq i \leq k$ and $1 \leq j \leq l$, there exist $\eta_{i,j,\varepsilon}, \zeta_{i,j ,\varepsilon},\mu_{i,j,\varepsilon},\nu_{i,j,\varepsilon} \in \gsigma$ such that:
\begin{enumerate}[(a)]
\item\label{it:word_oa} $O_{\varepsilon k,l}=\displaystyle \prod_{j=1}^l \left(\prod_{i=1}^k \eta_{i,j,\varepsilon} B_{2\omega k-2i+1, -j} \eta_{i,j,\varepsilon}^{-1} \zeta_{i,j,\varepsilon} B_{2\omega k-2i,j-1}^{-1} \zeta_{i,j,\varepsilon}^{-1} \right)^\varepsilon$.

\item\label{it:word_ob} $O_{\varepsilon k,-l}=\displaystyle \left( \prod_{j=1}^l \left(\prod_{i=1}^k \mu_{i,j,\varepsilon} B_{2\omega k-2i +1,l-j}  \mu_{i,j,\varepsilon}^{-1} \nu_{i,j,\varepsilon} B_{2\omega k-2i,-l+j-1}^{-1} \nu_{i,j,\varepsilon}^{-1} \right)^\varepsilon \right)^{-1}$.
\end{enumerate}
\end{lemma}

By taking $\varepsilon=1$ or $-1$ in parts~(\ref{it:word_oa}) and~(\ref{it:word_ob}) of Lemma~\ref{lem:word_o} and Abelianising, one may check that the formula for $\widetilde{O}_{k,l}$ given in Proposition~\ref{prop:words} is correct.

\begin{proof}[Proof of Lemma~\ref{lem:word_o}]
If $k=0$ or $l=0$, then clearly $O_{k,l}=\id$. So suppose that $k,l>0$, and let $\varepsilon \in \{-1,1\}$. We start by considering the case $l=1$. If $k=\varepsilon=1$ then:
\begin{align*}
O_{1,1}&= v^2 u v^{-2} u^{-1} = v u^{-1} u v u v^{-1} u v^{-1} v u^{-1} v^{-1} u^{-1} =  v u^{-1} B u v^{-1} B^{-1} = B_{1,-1} B_{0,0}^{-1},
\end{align*}
and so the expression for $O_{1,1}$ given in~(\ref{it:word_oa}) is correct by taking $\eta_{1,1,1}=\zeta_{1,1,1}=\id$. We now suppose that the expression for $O_{k,1}$ given in~(\ref{it:word_oa}) holds for some $k\geq 1$, where $\eta_{i,1,1}=\zeta_{i,1,1}=\id$ for all $1\leq i\leq k$. Then:
\begin{align*}
O_{k+1,1}&= v^{2(k+1)} u v^{-2(k+1)} u^{-1}	= v^2 v^{2k} u v^{-2k} u^{-1} v^{-2} v^2 u v^{-2} u^{-1}=v^2 O_{k,1} v^{-2} O_{1,1}\\
& =  v^2 \left(\prod_{i=1}^k B_{2k-2i+1,-1} B_{2k-2i,0}^{-1} \right) v^{-2} B_{1,-1} B_{0,0}^{-1} \\
&=  \left(\prod_{ i=1}^k B_{2(k+1)-2i+1,-1} B_{2(k+1)-2i,0}^{-1} \right) B_{1,-1} B_{0,0}^{-1}  =  \prod_{i=1}^{k+1} B_{2(k+1)-2i+1,-1} B_{2(k+1)-2i,0}^{-1}.
\end{align*}
By induction, it follows that that the expression for $O_{k,1}$ given in~(\ref{it:word_oa}) is correct for all $k\geq 1$, where $\eta_{i,1,1}=\zeta_{i,1,1}=\id$ for all $1\leq i\leq k$. Using this, for all $k\geq 1$, we obtain:
\begin{align*}
O_{-k,1}&= v^{-2k} u v^{2k} u^{-1}=v^{-2k}(v^{2k} u v^{-2k} u^{-1} )^{-1} v^{2k}=v^{-2k} O_{k,1}^{-1} v^{2k}\\
&=  v^{-2k} \left(\prod_{i=1}^k B_{2k -2i+1,-1} B_{2k-2i,0}^{-1} \right)^{-1} v^{2k}
=  \left(\prod_{i=1}^k B_{-2i+1,-1} B_{-2i,0}^{-1} \right)^{-1}.
\end{align*}
Hence the expression for $O_{-k,1}$ given in~(\ref{it:word_oa}) is correct for all $k\geq 1$, where $\eta_{i,1,-1}=\zeta_{i,1,-1}=\id$ for all $1\leq i\leq k$. We now suppose that the expression for $O_{\varepsilon k,l}$ given in~(\ref{it:word_oa}) is correct for some $l\geq 1$ and all $k\geq 1$ and $\varepsilon \in \{ 1,-1\}$, where $\eta_{i,j,\varepsilon}, \zeta_{i,j ,\varepsilon}\in \gsigma$ for all $1 \leq i \leq k$ and $1 \leq j \leq l$. We just saw that this is the case if $l=1$. We now study the case $l+1$. By Lemma~\ref{lem:homo_gsigma}(\ref{it:homo_gsigmac}), for all $1 \leq i \leq k$, there exist $\eta_{i, l+1, \varepsilon},\zeta_{i, l+1,\varepsilon} \in \gsigma$ such that $u^l B_{2\omega k-2i +1, -1} u^{-l}=\eta_{i,l+1,\varepsilon} B_{2\omega k-2i+1,-l-1} \eta_{i,l+1,\varepsilon}^{-1}$ and $u^l B_{2\omega k -2i,0} u^{-l}=\zeta_{i,l+1,\varepsilon} B_{2\omega k -2i,l} \zeta_{i,l+1,\varepsilon}^{-1}$. Thus:
\begin{align*}
O_{\varepsilon k,l+1}=& v^{2 \varepsilon k}u^{l+1} v^{-2 \varepsilon k}u^{-l-1} = v^{ 2 \varepsilon k}u^l v^{-2 \varepsilon k} u^{-l} u^l v^{2 \varepsilon k} u v^{-2 \varepsilon k}u^{-1} u^{-l}\\
=&  O_{\varepsilon k,l}u^{l} O_{\varepsilon k,1} u^{-l} \\
=&\prod_{j=1}^l \left(\prod_{i=1}^k \eta_{i,j,\varepsilon} B_{2\omega k-2i+1, -j} \eta_{i,j,\varepsilon}^{-1} \zeta_{i,j,\varepsilon} B_{2\omega k-2i,j-1}^{-1} \zeta_{i,j,\varepsilon}^{-1} \right)^\varepsilon \\
& u^l  \left( \prod_{i=1}^k B_{2\omega k-2i+1,-1} B_{2\omega  k-2 i,0}^{-1} \right)^\varepsilon u^{-l}\\
=& \prod_{j=1}^l \left(\prod_{i=1}^k \eta_{i,j,\varepsilon} B_{2\omega k-2i+1, -j} \eta_{i,j,\varepsilon}^{-1} \zeta_{i,j,\varepsilon} B_{2\omega k-2i,j-1}^{-1} \zeta_{i,j,\varepsilon}^{-1} \right)^\varepsilon \ldotp\\
& \left( \prod_{i=1}^k \eta_{i,l+1,\varepsilon} B_{2\omega k-2i+1,-l-1} \eta_{i,l+1,\varepsilon}^{-1} \zeta_{i,l+1,\varepsilon} B_{2\omega k -2i,l} \zeta_{i,l+1,\varepsilon}^{-1}  \right)^\varepsilon \\
=& \prod_{j=1}^{l+1} \left(\prod_{i=1}^k \eta_{i,j,\varepsilon} B_{2\omega k-2i+1, -j} \eta_{i,j,\varepsilon}^{-1} \zeta_{i,j,\varepsilon} B_{2\omega k-2i,j-1}^{-1} \zeta_{i,j,\varepsilon}^{-1} \right)^\varepsilon.
\end{align*}
By induction, it follows that~(\ref{it:word_oa}) holds for all $k,l\geq 1$ and $\varepsilon \in \{ 1,-1\}$. We now use part~(\ref{it:word_oa}) to prove part~(\ref{it:word_ob}). Let $k,l\geq 1$ and $\varepsilon \in \{ 1,-1\}$. By Lemma~\ref{lem:homo_gsigma}(\ref{it:homo_gsigmac}), for all $1 \leq i \leq k$ and $1 \leq j \leq l$, there exist $\mu_{i,j,\varepsilon},\nu_{i,j,\varepsilon}\in \gsigma$ such that $u^{-l} \eta_{i,j,\varepsilon} B_{2wk-2i +1,-j} \eta_{i,j,\varepsilon}^{-1} u^l= \mu_{i,j,\varepsilon} B_{2\omega k-2i +1,l-j}  \mu_{i,j,\varepsilon}^{-1}$ and $u^{-l} \zeta_{i,j,\varepsilon} B_{ 2wk-2i,j-1}  \zeta_{i,j,\varepsilon}^{-1} u^l = \nu_{i,j,\varepsilon} B_{ 2\omega k-2i,-l+j-1} \nu_{i,j,\varepsilon}^{-1}$.
\begin{align*}
O_{\varepsilon k,-l}&= v^{ 2 \varepsilon k} u^{-l} v^{- 2 \varepsilon k}u^l= u^{-l}(v^{ 2 \varepsilon k} u^l v^{- 2 \varepsilon k}u^{-l} )^{-1} u^l= u^{-l} O_{ 2 \varepsilon k,l}^{-1} u^l\\
&=  u^{-l} \left( \prod_{j=1}^l \left(\prod_{i=1}^k \eta_{i,j,\varepsilon} B_{2\omega k-2i+1, -j} \eta_{i,j,\varepsilon}^{-1} \zeta_{i,j,\varepsilon} B_{2\omega k-2i,j-1}^{-1} \zeta_{i,j,\varepsilon}^{-1} \right)^\varepsilon\right)^{-1} u^l\\
&=  \left( \prod_{j=1}^l \left(\prod_{i=1}^k \mu_{i,j,\varepsilon} B_{2\omega k-2i +1,l-j}  \mu_{i,j,\varepsilon}^{-1} \nu_{i,j,\varepsilon} B_{ 2\omega k-2i,-l+j-1}^{-1} \nu_{i,j,\varepsilon}^{-1} \right)^\varepsilon \right)^{-1},
\end{align*}
which proves part~(\ref{it:word_ob}). 
\end{proof}

%%%%%%%%%%%%%%%%%%%%%%%%%%%%%%%%%%%%%
% Proof of Theorem~\ref{th:BORSUK_TAU_1}
%%%%%%%%%%%%%%%%%%%%%%%%%%%%%%%%%%%%%

\section{Proof of Theorem~\ref{th:BORSUK_TAU_1}}\label{sec:borsuk_1}

Let $\alpha \in [ \torus,* ; \klein,*]$ and $\beta \in [ \torus,\klein]$. With the notation of~\cite[Theorem~4]{GonGuaLaa}, suppose that $\alpha_\mathcal{F}=\beta$. According to~\cite[Theorem~7]{GonGuaLaa}, the pointed homotopy class $\alpha$ has the Borsuk-Ulam property with respect to the free involution $\tau_1$ if and only if the homotopy class $\beta$ has the Borsuk-Ulam property with respect to $\tau_1$. So to obtain Theorem~\ref{th:BORSUK_TAU_1}, it suffices to prove the statement for pointed homotopy classes. Before doing so, we give an algebraic criterion, similar to that of~\cite[Lemma~22]{GonGuaLaa}, to decide whether a pointed homotopy class has the Borsuk-Ulam property with respect to $\tau_1$.

\begin{lemma}\label{lem:algebra_tau_1}
Let $\alpha \in [\torus,* ; \klein, *]$ be a pointed homotopy class. Then $\alpha$ does not have the Borsuk-Ulam property with respect to $\tau_1$ if and only if there exist $a,b \in P_2(\klein)$ such that:
\begin{enumerate}[(i)]
\item\label{it:algebra_tau_1a} $a l_\sigma(b)=b a$.
\item\label{it:algebra_tau_1b} $(p_1)_\#(a l_\sigma(a))= \alpha_\# (1,0)$.
\item\label{it:algebra_tau_1c} $(p_1)_\# (b)=\alpha_\# (0,1)$.
\end{enumerate}
\end{lemma}

\begin{proof}
The result may be obtained in a manner similar to that of \cite[Lemma~22]{GonGuaLaa}, using Proposition~\ref{prop:presentation_p2} instead of~\cite[Theorem~12]{GonGuaLaa}.
\end{proof}

\begin{corollary}\label{cor:reduction_tau_1}
Let $\alpha,\alpha' \in [ \torus,* ; \klein,*]$ be pointed homotopy classes, and suppose that:
\begin{equation*}
\text{$\alpha_\#: \begin{cases}
	 (1,0) \mapsto(r_1,s_1)\\
	 (0,1) \mapsto(r_2,s_2)
	\end{cases}$ and $\alpha'_\#: \begin{cases}
	 (1,0) \mapsto(r_1,s_1')\\
	 (0,1) \mapsto(r_2,s_2') 
	\end{cases}$}
\end{equation*}
for some $r_1, r_2,s_1,s_1',s_2,s_2' \in \z$. If $s_1 \equiv s_1' \bmod{4}$ and $s_2 \equiv s_2' \bmod{2}$ then $\alpha$ has the Borsuk-Ulam property with respect to $\tau_1$ if and only if $\alpha'$ does.
\end{corollary}

\begin{proof}
Since the statement is symmetric with respect to $\alpha$ and $\alpha'$, it suffices to show that if $\alpha$ does not have the Borsuk-Ulam property then neither does $\alpha'$. If $\alpha$ does not have the Borsuk-Ulam property, there exist $a,b \in P_2(\klein)$ satisfying~(\ref{it:algebra_tau_1a})--(\ref{it:algebra_tau_1c}) of Lemma~\ref{lem:algebra_tau_1}. By hypothesis, there exist $k_1,k_2 \in \z$ such that $s'_1=s_1+4 k_1$ and $s'_2=s_2+2 k_2$. Let $a'=a (\id ; 0,2 k_1)$ and $b'=b (\id ; 0,2 k_2)$ in $P_2(\klein)$. It suffices to show that $a'$ and $b'$ satisfy~(\ref{it:algebra_tau_1a})--(\ref{it:algebra_tau_1c}) of Lemma~\ref{lem:algebra_tau_1}.Using Proposition~\ref{prop:presentation_p2}, one may check that the centre of $B_2(\klein)$ is the subgroup $\langle(\id ; 0,2)\rangle$. Thus:
\begin{align*}
a' l_\sigma(b')&= a (\id ; 0,2 k_1) l_\sigma(b (\id ; 0,2 k_2)) = a l_\sigma(b)  (\id ; 0,2k_1+2 k_2)  \stackrel{\text{(\ref{it:algebra_tau_1a})}}{=}  b a (\id ; 0,2k_1+2 k_2)=b' a',\\
(p_1)_\#(a' l_\sigma(a')) &= (p_1)_\#(a (\id ; 0,2 k_1)l_\sigma (a (\id ; 0,2 k_1)))=(p_1)_\#(a l_\sigma (a) (\id; 0, 4 k_1)) \\ 
& \stackrel{\text{(\ref{it:algebra_tau_1b})}}{=} (r_1, s_1)(0,4 k_1)=(r_1,s_1')=\alpha'_\# (1,0), \text{ and} \\
(p_1)_\#(b')&=(p_1)_\#(b(\id ; 0,2 k_2)) \stackrel{\text{(\ref{it:algebra_tau_1c})}}{=} (r_2,s_2)(0, 2 k_2)=(r_2,s_2')=\alpha'_\# (0,1),
\end{align*}
which proves the corollary.
\end{proof}

\begin{remark}\label{rem:summary}
Let $\alpha \in [ \torus, * ; \klein,*]$ be a pointed homotopy class, and let $\alpha_\# : \pi_1(\torus)\to \pi_1(\klein)$ be the homomorphism described in~\cite[Theorem~4]{GonGuaLaa}, and that is of one of the four types given in Proposition~\ref{prop:set_homotopy}.
\begin{enumerate}[(a)]
\item\label{it:summarya}  Suppose that $\alpha_\#$ is of Type~1, 2 or~3, and let $i\in \{0,1\}$, $s_{1}$ and $s_{2}$ be the integers that appear in the description of $\alpha_\#$ in Proposition~\ref{prop:set_homotopy}. By Proposition~\ref{prop:reduced_cases} and Corollary~\ref{cor:reduction_tau_1}, $\alpha$ has the Borsuk-Ulam property with respect to $\tau_1$ if and only if $\alpha'$ does, where $\alpha'\in [ \torus, * ; \klein,*]$ satisfies:
\begin{enumerate}[(i)]
\item\label{it:summaryai} $\alpha'_{\#}(1,0)=(0, 2s_{1}+1 \bmod{4})$ and $\alpha'_{\#}(0,1)=(0, j)$ if $\alpha_\#$ is of Type~1 (in which case $j=0$), or is of Type~2 (in which case $j=1$).
\item\label{it:summaryaii} $\alpha'_{\#}(1,0)=(0, 2s_{1} \bmod{4})$ and $\alpha'_{\#}(0,1)=(0, 1)$ if $\alpha_\#$ is of Type~3.
\end{enumerate}
So for each of Types~1, 2 and~3, there are two cases to consider, $s_{1}=0$, and $s_{1}=1$. 

\item\label{it:summaryb} Suppose that $\alpha_\#$ is of Type~4, and let $r_{1},r_{2},s_{1}$ and $s_{2}$ be the integers that appear in the description of $\alpha_\#$ in Proposition~\ref{prop:set_homotopy}, where $r_{1}\geq 0$. By Proposition~\ref{prop:reduced_cases} and Corollary~\ref{cor:reduction_tau_1}, $\alpha$ has the Borsuk-Ulam property with respect to $\tau_1$ if and only if $\alpha'$ does, where $\alpha'\in [ \torus, * ; \klein,*]$ satisfies $\alpha'_{\#}(1,0)=(r_{1}, 2s_{1} \bmod{4})$ and $\alpha'_{\#}(0,1)=(r_{2}, 0)$. So for each pair of integers $(r_{1},r_{2})$, where $r_{1}\geq 0$, there are two cases to consider, $s_{1}=0$, and $s_{1}=1$.
\end{enumerate}
\end{remark}

To prove Theorem~\ref{th:BORSUK_TAU_1}, it suffices to study the cases described by Remark~\ref{rem:summary}. This will be carried out in Propositions~\ref{prop:tau_1_cases_134}--\ref{prop:tau_1_case_6} below. Part of Proposition~\ref{prop:tau_1_cases_134}  (resp.\ Proposition~\ref{prop:tau_1_case_2}) treats the cases of Remark~\ref{rem:summary}(\ref{it:summarya})(\ref{it:summaryai}) (resp. Remark~\ref{rem:summary}(\ref{it:summarya})(\ref{it:summaryaii})), and part of Proposition~\ref{prop:tau_1_cases_134} and Propositions~\ref{prop:tau_1_case_5} and~\ref{prop:tau_1_case_6} deal with the cases of Remark~\ref{rem:summary}(\ref{it:summaryb}). In each case, we will make use of Proposition~\ref{prop:presentation_p2} and its notation, as well as the commutative diagram~(\ref{eq:diag_rho_g}).

\begin{proposition}\label{prop:tau_1_cases_134}
Let:
\begin{equation*}
\Sigma=\{ (0,2s+1,0,j) \, \vert \, j,s\in \{0,1\} \} \cup \{ (r_{1},0,r_{2},0) \, \vert \, r_{1},r_{2}\in \z,\, r_{1}\geq 0 \} \cup \{ (0,2,0,0)\}.
\end{equation*}
Let $\alpha \in [ \torus, * ; \klein,*]$ be a pointed homotopy class such that $\alpha_\#: \begin{cases}
(1,0) \mapsto(\rho,\gamma)\\
(0,1) \mapsto(\xi,\tau),
\end{cases}$ where $(\rho,\gamma,\xi,\tau)\in \Sigma$. Then $\alpha$ does not have the Borsuk-Ulam property with res\-pect to $\tau_1$.
\end{proposition}

\begin{proposition}\label{prop:tau_1_case_2}
Let $\alpha \in [ \torus, * ; \klein,*]$ be a pointed homotopy class such that $\alpha_\#: \begin{cases}
(1,0) \mapsto (0,2s)\\
(0,1) \mapsto (0,1),
\end{cases}$	 where $s \in \{ 0,1 \}$. Then $\alpha$ has the Borsuk-Ulam property with respect to $\tau_1$.
\end{proposition}

\begin{proposition}\label{prop:tau_1_case_5}
If $\alpha_\#: \begin{cases}
(1,0) \mapsto(r_1,2)\\
(0,1) \mapsto(r_2,0),
\end{cases}$	 where $r_1, r_2 \in \mathbb{Z}$, $r_1 > 0$, and one of the following conditions holds:
\begin{enumerate}[(a)]
\item $r_{2}=0$. 
\item $r_2 \neq 0$ and $e(r_1) \leq e(r_2)$. 
\end{enumerate}
Then $\alpha$ does not have the Borsuk-Ulam property with respect to $\tau_1$.
\end{proposition}

\begin{proposition}\label{prop:tau_1_case_6}
Let $\alpha \in [ \torus, * ; \klein,*]$ be a pointed homotopy class such that $\alpha_\#: \begin{cases}
(1,0) \mapsto(r_1,2)\\
(0,1) \mapsto(r_2,0),
\end{cases}$	 where $r_1, r_2 \in \mathbb{Z}$, $r_{1}\geq 0$, $r_2 \neq 0$, and one of the following conditions holds:
\begin{enumerate}[(a)]
\item $r_{1}=0$.
\item $r_{1}>0$, and $e(r_1) > e(r_2)$. 
\end{enumerate}
Then $\alpha$ has the Borsuk-Ulam property with respect to $\tau_1$.
\end{proposition}

The following lemma will be used in the proofs of some of these propositions.

\begin{lemma}\label{lem:gencalc}
Let $a,b\in P_2(\klein)$. Then there exist $x,y \in \gsigma$ and $a_{i},b_{i},m_{i},n_{i}\in \z$, where $i\in \{1,2\}$, such that:
\begin{equation}\label{eq:abdecomp}
\text{$a =(u^{a_1} v^{a_2} x ; m_1,n_1)$ and $b =(u^{b_1} v^{b_2} y ; m_2,n_2)$.}
\end{equation}
Suppose further that $a$ and $b$ satisfy the relation of Lemma~\ref{lem:algebra_tau_1}(\ref{it:algebra_tau_1a}). Then:
\begin{equation}\label{eq:n1n2}
\text{$b_{2}=0$ and $(1+(-1)^{\delta_{n_{1}}+1})m_{2}= (1+(-1)^{\delta_{n_{2}}+1})m_{1}+(-1)^{\delta_{n_{1}}}b_{1}$},
\end{equation}
so $b_{1}$ is even, and:
\begin{multline}\label{eq:basiclsigma}
u^{b_{1}} y B^{m_2-\delta_{n_2}} u^{a_{1}\varepsilon_{n_{2}}} (B^{\delta_{n_{2}}}vu^{-2m_{2}})^{a_{2}} B^{\delta_{n_2}-m_2} \theta(m_2,\delta_{n_2})(x)=\\ u^{a_1} v^{a_2} x B^{m_1-\delta_{n_1}} (B^{\varepsilon_{n_{1}}} u^{-\varepsilon_{n_{1}}})^{b_{1}} B^{-\varepsilon_{n_{1}}b_{1}+\delta_{n_{1}}-m_{1}} \theta(m_{1}+\varepsilon_{n_{1}}b_{1}, \delta_{n_{1}})(\rho(y))B^{\delta_{n_{2}}\varepsilon_{n_{1}}}.
\end{multline}
\end{lemma}

\begin{proof}
Let $a,b\in P_2(\klein)$. Proposition~\ref{prop:normal_form} implies that there exist $x,y \in \gsigma$ and $a_{i},b_{i},m_{i},n_{i}\in \z$, where $i\in \{1,2\}$, for which~(\ref{eq:abdecomp}) holds. First, we have:
\begin{align}
ba&= (u^{b_1} v^{b_2} y ; m_2,n_2) (u^{a_1} v^{a_2} x ; m_1,n_1)\notag\\
& = (u^{b_1} v^{b_2} y \theta(m_2,\delta_{n_2})(u^{a_1} v^{a_2} x); m_{2}+(-1)^{\delta_{n_{2}}}m_{1}, n_{2}+n_{1}).\label{eq:ba}
\end{align}
Now:
\begin{align}
l_{\sigma}(b)=& ((Bu^{-1})^{b_{1}}B^{-b_{1}}; b_{1},0)((uv)^{-b_{2}} (uB)^{\delta_{b_{2}}};0,b_{2}) (\rho(y); 0,0) (\id;m_{2},0) (B^{\delta_{n_{2}}};0,n_{2})\notag\\
=& ((Bu^{-1})^{b_{1}}B^{-b_{1}}; b_{1},0)((uv)^{-b_{2}} (uB)^{\delta_{b_{2}}};0,b_{2})(\rho(y)B^{\delta_{n_{2}}}; m_{2},n_{2})\notag\\
=&((Bu^{-1})^{b_{1}}B^{-b_{1}} \theta(b_{1},0)((uv)^{-b_{2}} (uB)^{\delta_{b_{2}}}); b_{1},b_{2})(\rho(y)B^{\delta_{n_{2}}}; m_{2},n_{2})\notag\\
=&((Bu^{-1})^{b_{1}}B^{-b_{1}} \theta(b_{1},0)((uv)^{-b_{2}} (uB)^{\delta_{b_{2}}}) \theta(b_{1},\delta_{b_{2}})(\rho(y)B^{\delta_{n_{2}}}); \notag\\
& b_{1}+(-1)^{\delta_{b_{2}}}m_{2}, b_{2}+n_{2}).\label{eq:lsigb}
\end{align}
Thus:
\begin{align}
(p_{1})_{\#}(al_{\sigma}(b))&=(m_1,n_1) (b_{1}+(-1)^{\delta_{b_{2}}}m_{2}, b_2+n_{2})\notag\\
&= (m_{1}+(-1)^{\delta_{n_{1}}}b_{1}+(-1)^{\delta_{n_{1}}+\delta_{b_{2}}}m_{2}, n_{1}+b_{2}+n_{2}).\label{eq:galb}
\end{align}
Suppose that $a$ and $b$ satisfy the relation of~Lemma~\ref{lem:algebra_tau_1}(\ref{it:algebra_tau_1a}). It follows from~(\ref{eq:ba}) and (\ref{eq:galb}) that $b_{2}=0$, and then that $m_{2}+(-1)^{\delta_{n_{2}}}m_{1}= m_{1}+(-1)^{\delta_{n_{1}}}b_{1}+(-1)^{\delta_{n_{1}}}m_{2}$, which yields~(\ref{eq:n1n2}). This implies that $b_{1}$ is even. We now expand and simplify the remaining parts of~(\ref{eq:ba}) and~(\ref{eq:lsigb}):
\begin{align}\label{eq:projba}
p_{F}(ba)&= u^{b_{1}} y \theta(m_2,\delta_{n_2})(u^{a_1} v^{a_2} x)\notag\\
&= u^{b_{1}} y B^{m_2-\delta_{n_2}} u^{a_{1}\varepsilon_{n_{2}}} (B^{\delta_{n_{2}}}vu^{-2m_{2}})^{a_{2}} B^{\delta_{n_2}-m_2}\theta(m_2,\delta_{n_2})(x),
\end{align}
and
\begin{align*}
p_{F}(l_{\sigma}(b))&= (Bu^{-1})^{b_{1}} B^{-b_{1}} \theta(b_{1},0)(\rho(y)B^{\delta_{n_{2}}})= (Bu^{-1})^{b_{1}} B^{-b_{1}} \theta(b_{1},0)(\rho(y))B^{\delta_{n_{2}}}.
\end{align*}
Hence:
\begin{align}
p_{F}(al_{\sigma}(b))=& u^{a_1} v^{a_2} x \theta(m_1,\delta_{n_1})((Bu^{-1})^{b_{1}} B^{-b_{1}} \theta(b_{1},0)(\rho(y)) B^{\delta_{n_{2}}})\notag\\
=& u^{a_1} v^{a_2} x B^{m_1-\delta_{n_1}} (B^{\varepsilon_{n_{1}}} u^{-\varepsilon_{n_{1}}})^{b_{1}} B^{-\varepsilon_{n_{1}}b_{1}+\delta_{n_{1}}-m_{1}} \theta(m_{1}+\varepsilon_{n_{1}}b_{1}, \delta_{n_{1}})(\rho(y)) B^{\delta_{n_{2}}\varepsilon_{n_{1}}}.\label{eq:projalsigma}
\end{align}
Equation~(\ref{eq:basiclsigma}) then follows from the hypothesis, and equations~(\ref{eq:projba}) and~(\ref{eq:projalsigma}).
\end{proof}

\begin{proof}[Proof of Proposition~\ref{prop:tau_1_cases_134}]
Let $\alpha \in [ \torus, * ; \klein,*]$ be a pointed homotopy class such that $\alpha_\#(1,0) =(\rho,\gamma)$ and $\alpha_\#(0,1) =(\xi,\tau)$, where $(\rho,\gamma,\xi,\tau)\in \Sigma$. Note that either $\rho=\xi=0$ or $\gamma=\tau=0$. In what follows, we will use the identities $\delta_{\delta_{q}}=\delta_{q}$ and $\varepsilon_{\delta_{q}}=\varepsilon_{q}$ for all $q\in \z$. Let $a,b\in P_{2}(\klein)$ be such that %\comj{one may check that these are exactly the $a$ and $b$ given in each of the three Propositions~\ref{prop:tau_1_case_1}
%,~\ref{prop:tau_1_case_3} and~\ref{prop:tau_1_case_4}, now combined into Proposition~\ref{prop:tau_1_cases_134})~--~there was a small error in %the previous version, this has (hopefully) been corrected.} 
$a=(u^{\delta_{\rho}} v^{\delta_{\gamma}} B^{\delta_{\gamma}(\gamma-\delta_{\gamma})/2}; (\rho-\delta_{\rho})/2, (\gamma-\delta_{\gamma})/2)$ and $b=(B^{-\delta_{\rho}\xi}; \xi,\tau)$. With respect to the notation of~(\ref{eq:abdecomp}), $x=B^{\delta_{\gamma}(\gamma-\delta_{\gamma})/2}$ and $y=B^{-\delta_{\rho}\xi}$. To prove the result, we show that conditions~(\ref{it:algebra_tau_1a})--(\ref{it:algebra_tau_1c}) of Lemma~\ref{lem:algebra_tau_1} are satisfied. Clearly, $(p_{1})_{\#}(b)=(\xi,\tau)=\alpha_{\#}(0,1)$. Further, by taking $b=a$ in~(\ref{eq:galb}), we have:
\begin{equation}\label{eq:ala134}
(p_{1})_{\#}(a l_{\sigma}(a))= \left(\bigl(1+(-1)^{\delta_{(\gamma-\delta_{\gamma})/2}+\delta_{\gamma}}\bigr)\frac{(\rho-\delta_{\rho})}{2} + \varepsilon_{(\gamma-\delta_{\gamma})/2} \delta_{\rho}, \gamma \right). 
\end{equation}
If $\rho=\xi=0$ then the first coordinate of the right-hand side of~(\ref{eq:ala134}) is equal to zero, while if $\gamma=\tau=0$, this coordinate is equal to $\rho$. In both cases, we conclude that $(p_{1})_{\#}(a l_{\sigma}(a))=(\rho,\gamma)=\alpha_{\#}(1,0)$. Hence conditions~(\ref{it:algebra_tau_1b}) and~(\ref{it:algebra_tau_1c}) of Lemma~\ref{lem:algebra_tau_1} are satisfied. It remains to check condition~(\ref{it:algebra_tau_1a}). Note that in the proof of Lemma~\ref{lem:gencalc}, the only condition that we have applied to obtain equations~(\ref{eq:projba}) and~(\ref{eq:projalsigma}) is that $b_{2}=0$. But this coefficient is zero in our case, and so these equations are also satisfied here. So using~(\ref{eq:ba}),~(\ref{eq:projba}) and Proposition~\ref{prop:presentation_p2}, we see that:
\begin{align}
ba=& \left(B^{-\delta_{\rho}\xi} \theta(\xi,\delta_{\tau})(u^{\delta_{\rho}} v^{\delta_{\gamma}} B^{\delta_{\gamma}(\gamma-\delta_{\gamma})/2}); \xi+\varepsilon_{\tau}\frac{(\rho-\delta_{\rho})}{2}, \tau+\frac{(\gamma-\delta_{\gamma})}{2}\right)\notag\\
=&\left(B^{-\delta_{\rho}\xi} B^{\xi-\delta_{\tau}}u^{\delta_{\rho}\varepsilon_{\tau}} (B^{\delta_{\tau}}vu^{-2\xi})^{\delta_{\gamma}} B^{\delta_{\tau}-\xi} B^{\varepsilon_{\tau} \delta_{\gamma}(\gamma-\delta_{\gamma})/2}; \xi+\varepsilon_{\tau} \frac{(\rho-\delta_{\rho})}{2}, \tau+\frac{(\gamma-\delta_{\gamma})}{2}\right).
\end{align}
In a similar manner, by~(\ref{eq:galb}) and~(\ref{eq:projalsigma}), we see that:
\begin{align}
al_{\sigma}(b)=& \left( u^{\delta_{\rho}} v^{\delta_{\gamma}} B^{\delta_{\gamma}(\gamma-\delta_{\gamma})/2+(\delta_{\tau}-\delta_{\rho}\xi)\varepsilon_{(\gamma-\delta_{\gamma})/2}}; \frac{(\rho-\delta_{\rho})}{2}+\varepsilon_{(\gamma-\delta_{\gamma})/2}\xi, \tau+\frac{(\gamma-\delta_{\gamma})}{2}\right).
\end{align}
If $\rho=\xi=0$ (resp.\ $\gamma=\tau=0$) then $(p_{1})_{\#}(ba)=(0,\tau+(\gamma-\delta_{\gamma})/2)=(p_{1})_{\#}(al_{\sigma}(b))$ (resp.\ $(p_{1})_{\#}(ba)=(\xi+(\rho-\delta_{\rho})/2,0)=(p_{1})_{\#}(al_{\sigma}(b))$). So it remains to show that $p_{F}(ba)=p_{F}(al_{\sigma}(b))$. 

\begin{enumerate}[(a)]
\item If $\rho=\xi=0$, $p_{F}(ba)= B^{-\delta_{\tau}} (B^{\delta_{\tau}}v)^{\delta_{\gamma}} B^{\delta_{\tau}} B^{\varepsilon_{\tau} \delta_{\gamma}(\gamma-\delta_{\gamma})/2}$ and $p_{F}(al_{\sigma}(b))= v^{\delta_{\gamma}} B^{\delta_{\gamma}(\gamma-\delta_{\gamma})/2+\delta_{\tau}\varepsilon_{(\gamma-\delta_{\gamma})/2}}$. 
\begin{enumerate}[(i)]
\item Suppose that $\delta_{\gamma}=0$. Since $(\rho,\gamma,\xi,\tau)\in \Sigma$, it follows that $\delta_{\tau}=0$, and thus $p_{F}(ba)=p_{F}(al_{\sigma}(b))=\id$.

\item Suppose that $\delta_{\gamma}=1$. Then $p_{F}(ba)= v B^{\delta_{\tau}+\varepsilon_{\tau} (\gamma-1)/2}$ and $p_{F}(al_{\sigma}(b))= v B^{(\gamma-1)/2+\delta_{\tau}\varepsilon_{(\gamma-1)/2}}$. 

If $\delta_{\tau}=0$ then $p_{F}(ba)= v B^{(\gamma-1)/2}= p_{F}(al_{\sigma}(b))$. So suppose that $\delta_{\tau}=1$. Since $(\rho,\gamma,\xi,\tau)\in \Sigma$, it follows that $(\gamma-1)/2\in \{ 0,1\}$, and one can check easily that $1-(\gamma-1)/2=(\gamma-1)/2+\varepsilon_{(\gamma-1)/2}$. It follows that $p_{F}(ba)=p_{F}(al_{\sigma}(b))$.
\end{enumerate}

\item If $\gamma=\tau=0$ then $p_{F}(ba)=B^{(1-\delta_{\rho})\xi} u^{\delta_{\rho}} B^{-\xi}$ and $p_{F}(al_{\sigma}(b))=u^{\delta_{\rho}}  B^{-\delta_{\rho}\xi}$, and one sees that $p_{F}(ba)=\id=p_{F}(al_{\sigma}(b))$ if $\delta_{\rho}=0$, and $p_{F}(ba)=uB^{-\xi}=p_{F}(al_{\sigma}(b))$ if $\delta_{\rho}=1$.\qed
\end{enumerate}
\end{proof}

In order to prove Proposition~\ref{prop:tau_1_case_2}, we will make use of the following lemma.

\begin{lemma}\label{lem:aux_1}
Let $\alpha \in [\torus,* ; \klein,*]$ be such that 
$\alpha_\#: \begin{cases}
(1,0) \mapsto(0,2s)\\
(0,1) \mapsto(0,1)
\end{cases}$
for some $s \in \{ 0,1 \}$. If $\alpha$ does not have the Borsuk-Ulam property with respect to $\tau_1$ then there exist $x,y \in \gsigmab$ and $m_1,n_1 \in \z$ that satisfy  the following equation in $\gsigmab$:
\begin{align}\label{eq:newabeq}
\mu_{1}(x)+\mu_{2}(y)=&\widetilde{I}_{2n_1-2s}-\widetilde{T}_{-2m_1,\delta_{n_1}}-\widetilde{O}_{n_1-s,-2m_1} -(m_1+\delta_{n_1}+\varepsilon_{n_1}) B_{0,0} -(m_1-\delta_{n_1}) B_{0,-2m_1} \notag \\
& -B_{2n_1-2s,-2\delta_{n_{1}+1}m_1} +B_{2n_1-2s,0},
\end{align}
where $\mu_1,\mu_2 : \gsigmab \to \gsigmab$ are the homomorphisms defined on the basis elements of $\gsigmab$ by:
\begin{align*}
\mu_1(B_{k,l})=&B_{k,l-2\varepsilon_{k}m_1}+B_{k,-l} \text{ and } \\
\mu_2(B_{k,l})=&\varepsilon_{k} \varepsilon_{\delta_{n_1}} B_{-k,\varepsilon_{\delta_{n_1}}\varepsilon_{k+1}l+2\delta_{k}m_1}-B_{k+2n_1-2s,l-2\varepsilon_{k}\delta_{n_{1}+1}m_1}.
\end{align*}
\end{lemma}

\begin{proof}
By hypothesis, there exist $a,b\in P_2(\klein)$ satisfying~(\ref{it:algebra_tau_1a})--(\ref{it:algebra_tau_1c}) of Lemma~\ref{lem:algebra_tau_1}.  Proposition~\ref{prop:normal_form} implies that $a$ and $b$ may be written in the form~(\ref{eq:abdecomp}). Lemma~\ref{lem:algebra_tau_1}(\ref{it:algebra_tau_1c}) implies that $(p_1)_\#(b)  =(0,1)$, so $(m_2,n_2)=(0,1)$. Using Lemma~\ref{lem:algebra_tau_1}(\ref{it:algebra_tau_1b}) and taking $b=a$ in~(\ref{eq:lsigb}), we obtain:
\begin{align*}%\label{eq:equality_BU_t2_k2_1_pont_3_b}
(0,2s) &=(p_1)_\#(a l_\sigma(a))=(m_1,n_1)(a_1+(-1)^{a_2} m_1,n_1+a_2)\\
%= (p_1)_\#(u^{a_1} v^{a_2} x ; m_1,n_1)(p_1)_\#(l_\sigma(u^{a_1} v^{a_2} x ; m_1,n_1 ))\\
%&=(m_1,n_1)(p_1)_\# (l_\sigma(u^{a_1} v^{a_2} x ; 0,0)) (p_1)_\#(l_\sigma(\id ; m_1,n_1))=(m_1,n_1)(a_1,a_2)(m_1,n_1)\\
%& =(m_1+(-1)^{n_1} a_1,n_1+a_2)(m_1,n_1) 
&=(m_1+(-1)^{n_1} a_1+(-1)^{n_1+a_2} m_1, 2n_1+a_2).
\end{align*}
It follows that $a_2=2s-2n_1$ is even and $a_1= -2\delta_{n_{1}+1}m_1$.  
Lemmas~\ref{lem:algebra_tau_1}(\ref{it:algebra_tau_1a}) and~\ref{lem:gencalc} then imply that $b_2=0$ and $b_{1}=-2\varepsilon_{n_{1}}m_{1}$. We now analyse~(\ref{eq:basiclsigma}), which holds because Lemma~\ref{lem:algebra_tau_1}(\ref{it:algebra_tau_1a}) does. 
The left-hand side of~(\ref{eq:basiclsigma}) is equal to $p_{F}(ba)$, and may be rewritten as: 
\begin{align}
p_{F}(ba)=& u^{b_{1}} y B^{-1} u^{-a_{1}} (B v)^{2s-2n_{1}} B \theta(0,1)(x)=u^{b_{1}} y B^{-1} u^{-a_{1}} B (vB)^{2s-2n_{1}}  \theta(0,1)(x)\notag\\
=& u^{b_{1}-a_{1}} v^{2s-2n_{1}} \ldotp v^{2n_{1}-2s} u^{a_{1}}  y u^{-a_{1}} v^{2s-2n_{1}} \ldotp v^{2n_{1}-2s} u^{a_{1}} B^{-1} u^{-a_{1}} v^{2s-2n_{1}} \ldotp \notag\\
& v^{2n_{1}-2s} B v^{2s-2n_{1}} \ldotp v^{2n_{1}-2s} (vB)^{2s-2n_{1}}  \theta(0,1)(x)\notag\\
=& u^{b_{1}-a_{1}} v^{2s-2n_{1}} c_{2n_{1}-2s,a_{1}}(y) B_{2n_{1}-2s,a_{1}}^{-1} B_{2n_{1}-2s,0} I_{2n_{1}-2s}\theta(0,1)(x),\label{eq:pfba}
\end{align}
and the right-hand side of~(\ref{eq:basiclsigma}) is equal to $p_{F}(al_{\sigma}(b))$, and may be rewritten as:
\begin{align}
p_{F}(al_{\sigma}(b))=& u^{a_1} v^{2s-2n_{1}} x B^{m_1-\delta_{n_1}} (B^{\varepsilon_{n_{1}}} u^{-\varepsilon_{n_{1}}})^{b_{1}} B^{m_{1}+\delta_{n_{1}}} \theta(-m_{1}, \delta_{n_{1}})(\rho(y))B^{\varepsilon_{n_{1}}}\notag\\
=& u^{2m_{1}+a_{1}} v^{2s-2n_{1}} [v^{2n_{1}-2s}, u^{-2m_{1}}] c_{0,-2m_{1}}(x) B_{0,-2m_{1}}^{m_1-\delta_{n_1}} \ldotp\notag\\
& u^{-2m_{1}} (B^{\varepsilon_{n_{1}}} u^{-\varepsilon_{n_{1}}})^{-2\varepsilon_{n_{1}}m_{1}} B_{0,0}^{m_{1}+\delta_{n_{1}}} \theta(-m_{1}, \delta_{n_{1}})(\rho(y))B_{0,0}^{\varepsilon_{n_{1}}}\notag\\
=& u^{2m_{1}+a_{1}} v^{2s-2n_{1}} O_{n_{1}-s,-2m_{1}} c_{0,-2m_{1}}(x) B_{0,-2m_{1}}^{m_1-\delta_{n_1}} T_{-2m_{1},\delta_{n_{1}}} B_{0,0}^{m_{1}+\delta_{n_{1}}}\ldotp\notag\\
& \theta(-m_{1}, \delta_{n_{1}})(\rho(y))B_{0,0}^{\varepsilon_{n_{1}}},\label{eq:pfasigmab}
\end{align}
Since $p_{F}(ba)=p_{F}(al_{\sigma}(b))$, and using the fact that $b_{1}-a_{1}=2m_{1}+a_{1}$, it follows by Abelianising~(\ref{eq:pfba}) and~(\ref{eq:pfasigmab}) that:
\begin{multline*}
 (c_{0,-2m_{1}})_{\ab}(x)-\theta(0,1)_{\ab}(x) +\theta(-m_{1}, \delta_{n_{1}})_{\ab} \circ \rho_{\ab}(y)-(c_{2n_{1}-2s,a_{1}})_{\ab}(y) = \\
 \widetilde{I}_{2n_{1}-2s}-\widetilde{T}_{-2m_{1},\delta_{n_{1}}}-\widetilde{O}_{n_{1}-s,-2m_{1}}-(m_{1}+\delta_{n_{1}}+\varepsilon_{n_{1}}) B_{0,0} -(m_1-\delta_{n_1}) B_{0,-2m_{1}} -B_{2n_{1}-2s,a_{1}} +B_{2n_{1}-2s,0}
\end{multline*}
in $\gsigmab$, where the projection of $x$ (resp.\ $y$) in $\gsigmab$ is also denoted by $x$ (resp.\ $y$). By~(\ref{eq:homo_gsigmab_theta})--(\ref{eq:homo_gsigmab_cpq}), for all $k,l\in \z$, we have:
\begin{equation*}
(c_{0,-2m_{1}} )_\ab(B_{k,l})- \theta(0,1)_\ab(B_{k,l})= B_{k,l-2\varepsilon_{k}m_{1}}+B_{k,-l}=\mu_1(B_{k,l})
\end{equation*}
and
\begin{align*}
\theta(-m_{1}, \delta_{n_{1}})_{\ab} \circ \rho_{\ab}(B_{k,l})-(c_{2n_{1}-2s,a_{1}})_{\ab}(B_{k,l}) &= 
\theta(-m_{1}, \delta_{n_{1}})_{\ab}(\varepsilon_{k}B_{-k,\varepsilon_{k+1}l})-B_{k+2n_{1}-2s,l+\varepsilon_{k}a_{1}}\\
&=\varepsilon_{k} \varepsilon_{\delta_{n_{1}}} B_{-k,\varepsilon_{\delta_{n_{1}}}\varepsilon_{k+1}l+2\delta_{k}m_{1}}-B_{k+2n_{1}-2s,l+\varepsilon_{k}a_{1}}\\
&=\mu_{2}(B_{k,l}).
\end{align*}
The result follows by noting that $a_1 = -2\delta_{n_{1}+1}m_1$.
\end{proof}

We are now able to complete the proof of Proposition~\ref{prop:tau_1_case_2}. 

\begin{proof}[Proof of Proposition~\ref{prop:tau_1_case_2}]
We argue by contradiction. Suppose that $\alpha$ does not have the Borsuk-Ulam property with respect to $\tau_1$. Then there exist $x,y\in \gsigmab$ that satisfy equation~(\ref{eq:newabeq}) given in the statement of Lemma~\ref{lem:aux_1}. Let $\xi : \gsigmab \to \ztwo$ be the homomorphism defined on the basis $\{ B_{k,l}\}_{k,l\in \z}$ of $\gsigmab$ by $\xi(B_{k,l})= \overline{1}$ for all $k,l\in \z$. From the definition of the maps $\mu_{1}$ and $\mu_{2}$, it follows that the left-hand side is sent to $\overline{0}$. By Proposition~\ref{prop:words}, $\xi(\widetilde{I}_{2n_{1}-2s})= \xi(\widetilde{T}_{-2m_{1},\delta_{n_{1}}})= \xi(\widetilde{O}_{n_{1}-s,-2m_{1}}) =\overline{0}$, and it follows that the right-hand side is sent to $\overline{\varepsilon}_{n_{1}}$, which is different from $\overline{0}$. This yields a contradiction. We conclude that $\alpha$ has the Borsuk-Ulam property with respect to $\tau_1$.
\end{proof}

\begin{proof}[Proof of Proposition~\ref{prop:tau_1_case_5}]
Let $\alpha \in [ \torus, * ; \klein,*]$ be a pointed homotopy class such that $\alpha_\#(1,0)=(r_1,2)$ and $\alpha_\#(0,1)=(r_2,0)$, where $r_1, r_2 \in \mathbb{Z}$ are such that $r_1 > 0$, and either $r_{2}=0$, or $r_2 \neq 0$ and $e(r_1) \leq e(r_2)$. Let $o(r_{1})=r_{1}/2^{e(r_{1})}$, and let $m=r_{2}/2^{e(r_{1})}$. Then $o(r_{1})>0$ is odd, and $m\in \z$ by hypothesis. Let $\alpha' \in [\torus,* ; \klein,*]$ be the homotopy class for which $\alpha'_\#(1,0)=(r_1,2 o(r_1))$ and $\alpha'_\#(0,1)=(r_2,2m)$. By Corollary~\ref{cor:reduction_tau_1}, to prove the result, it suffices to exhibit $a,b\in P_{2}(\klein)$ that satisfy conditions~(\ref{it:algebra_tau_1a})--(\ref{it:algebra_tau_1c}) of Lemma~\ref{lem:algebra_tau_1} for $\alpha'$. Let $c =(u^{2^{e(r_1)}} v^2; 0,0) \in P_2(\klein)$, and let $a=(c \sigma)^{o(r_1)} \sigma^{-1}$ and $b=(c \sigma)^{2m}$. Then $a,b \in P_2 (\klein)$, and by Proposition~\ref{prop:presentation_p2}, we see that $(p_1)_\#(l_\sigma(c))=(2^{e(r_1)},2)$. Now $a l_\sigma(b)=(c \sigma)^{o(r_1)} \sigma^{-1} \sigma (c \sigma)^{2m} \sigma^{-1}=(c \sigma)^{2m} (c \sigma)^{o(r_1)} \sigma^{-1}=ba$, so condition~(\ref{it:algebra_tau_1a}) of Lemma~\ref{lem:algebra_tau_1} is satisfied. Next, $a =((c \sigma )^2 )^{(o(r_1)-1)/2} c =(c l_\sigma (c) \sigma^2  )^{(o(r_1)-1)/2} c$, hence $l_\sigma(a) =(l_\sigma(c)\sigma^2 c)^{(o(r_1)-1)/2} l_\sigma(c)$, and since $(p_1)_\#(c)=(p_1)_\#(\sigma^2)=(0,0)$, it follows that $(p_1)_\#(a l_\sigma(a))=((p_1)_\#(l_\sigma(c)))^{o(r_1)}= (2^{e(r_1)},2)^{o(r_1)}=(2^{e(r_1)}o(r_1),2o(r_1))=(r_{1}, 2o(r_1))= \alpha_\#(1,0)$. So condition~(\ref{it:algebra_tau_1b}) of Lemma~\ref{lem:algebra_tau_1} holds. Finally, $b=(c \sigma c \sigma)^m =(c l_\sigma(c) \sigma^2)^m$, so $(p_1)_\#(b)=(p_1)_\#(l_\sigma(c))^m=(2^{e(r_1)},2)^{m}=(2^{e(r_1)}m,2m)=(r_2,2m)
= \alpha'_\# (0,1)$, and condition~(\ref{it:algebra_tau_1c}) of Lemma~\ref{lem:algebra_tau_1} is satisfied, which proves the proposition.
\end{proof}

The rest of this section is devoted to proving Proposition~\ref{prop:tau_1_case_6}.

\begin{lemma}\label{lem:lema_BU_t2_k2_1_pont_3_a}
Let $\alpha \in \left[ \torus,* ; \klein,* \right]$ be a homotopy class such that $\alpha_\# : \pi_1(\torus) \to \pi_1(\klein)$ satisfies $\alpha_\#(1,0 )=(r_1,2)$ and $\alpha_\# (0,1) =(r_2,0)$, where $r_{1},r_{2}\in \z$. With the notation of Proposition~\ref{prop:words}, if
$\alpha$ does not have the Borsuk-Ulam property, then there exist $x,y \in \gsigmab$ and $(m_1 ,n_1) \in \zsdz$ such that:
\begin{align}\label{eq:muxnuy}
\mu(x)+\nu(y)= & \widetilde{J}_{n_{1}-1,-2r_{2}} -\widetilde{O}_{n_{1}-1,2\delta_{n_{1}}r_{2}}- \widetilde{T}_{2\delta_{n_{1}} r_{2},\delta_{n_{1}}} +\\
& r_{2} B_{2(n_{1}-1),2\delta_{n_{1}+1} m_{1}-\varepsilon_{n_{1}}r_{1}}-(m_1-\delta_{n_1})B_{0,2\delta_{n_{1}}r_{2}} -(\delta_{n_{1}}(1-2r_{2})-m_{1}+r_{2})B_{0,0} \notag
\end{align}
in $\gsigmab$, where $\mu,\nu : \gsigmab \to \gsigmab$ are the homomorphisms defined on the elements of the basis $\{ B_{k,l}\}_{k,l\in \z}$ of $\gsigmab$ by:
\begin{align}
\mu(B_{k,l}) =&B_{k,l+2\varepsilon_{k}\delta_{n_{1}}r_{2}}-B_{k,l-2\delta_{k}r_{2}}\label{eq:defmu}\\
\nu(B_{k,l}) =&\varepsilon_{k} \varepsilon_{n_{1}} B_{-k,\varepsilon_{n_{1}}\varepsilon_{k+1}l-2\delta_{k}(m_{1}+2\delta_{n_{1}}r_{2})} \label{eq:defnu}-B_{k+2(n_{1}-1), l+\varepsilon_{k}(2\delta_{n_{1}+1}m_{1}-\varepsilon_{n_{1}}r_{1})}. 
\end{align}
\end{lemma}

\begin{proof}
Let $\alpha \in \left[ \torus,* ; \klein,* \right]$ be a homotopy class that does not have the Borsuk-Ulam property, and that is represented by the homomorphism $\alpha_\# : \pi_1(\torus) \to \pi_1(\klein)$ given by $\alpha_\#(1,0)=(r_1,2)$ and $\alpha_\# (0,1) =(r_2,0)$, where $r_{1},r_{2}\in \z$. Then there exist $a,b\in P_{2}(\klein)$ satisfying conditions~(\ref{it:algebra_tau_1a})--(\ref{it:algebra_tau_1c}) of Lemma~\ref{lem:algebra_tau_1}. We write these elements in the form of equation~(\ref{eq:abdecomp}). Condition~(\ref{it:algebra_tau_1c}) implies that $(m_{2},n_{2})=(r_{2},0)$, and taking $b=a$ in~(\ref{eq:galb}), condition~(\ref{it:algebra_tau_1b}) implies that $(r_{1},2)=(m_1(1+(-1)^{n_1+a_2})+(-1)^{n_1} a_1, 2n_1+a_2)$. It follows that $a_{2}=2(1-n_{1})$ and that $a_{1}=\varepsilon_{n_{1}}r_{1}-2\delta_{n_{1}+1} m_{1}$. By~(\ref{eq:n1n2}), we have $b_{2}=0$ and $(1+(-1)^{\delta_{n_{1}}+1})m_{2}= (1+(-1)^{\delta_{n_{2}}+1})m_{1}+(-1)^{\delta_{n_{1}}}b_{1}$, from which we see that $b_{1}=-2\delta_{n_{1}}r_{2}$ using the fact that $n_{2}=0$. Substituting this information into~(\ref{eq:projba}) and~(\ref{eq:projalsigma}), and using that fact that $\varepsilon_{n_{1}}\delta_{n_{1}}=-\delta_{n_{1}}$, we obtain:
\begin{align*}
p_{F}(ba)=& u^{-2\delta_{n_{1}}r_{2}} y B^{r_2} u^{\varepsilon_{n_{1}}r_{1}-2\delta_{n_{1}+1} m_{1}} (vu^{-2r_{2}})^{2(1-n_{1})} B^{-r_2} \theta(r_2,0)(x)\\
=& u^{\varepsilon_{n_{1}}r_{1}-2\delta_{n_{1}}r_{2}-2\delta_{n_{1}+1} m_{1}} v^{2(1-n_{1})} \ldotp v^{2(n_{1}-1)} u^{2\delta_{n_{1}+1} m_{1}-\varepsilon_{n_{1}}r_{1}} y B^{r_2} \ldotp\\
& u^{\varepsilon_{n_{1}}r_{1}-2\delta_{n_{1}+1} m_{1}} v^{2(1-n_{1})}\ldotp
v^{2(n_{1}-1)}(vu^{-2r_{2}})^{2(1-n_{1})} B_{0,0}^{-r_2} \theta(r_2,0)(x)\\
=& u^{\varepsilon_{n_{1}}r_{1}-2\delta_{n_{1}}r_{2}-2\delta_{n_{1}+1} m_{1}} v^{2(1-n_{1})} c_{2(n_{1}-1),2\delta_{n_{1}+1} m_{1}-\varepsilon_{n_{1}}r_{1}}(y) \ldotp\\
& B_{2(n_{1}-1),2\delta_{n_{1}+1} m_{1}-\varepsilon_{n_{1}}r_{1}}^{r_{2}}
J_{n_{1}-1,-2r_{2}} B_{0,0}^{-r_2} \theta(r_2,0)(x),\\
p_{F}(al_\sigma(b))=& u^{\varepsilon_{n_{1}}r_{1}-2\delta_{n_{1}+1} m_{1}} v^{2(1-n_{1})} x B^{m_1-\delta_{n_1}} (B^{\varepsilon_{n_{1}}} u^{-\varepsilon_{n_{1}}})^{2\varepsilon_{n_{1}}\delta_{n_{1}} r_{2}} \ldotp\\
& B^{\delta_{n_{1}}(1-2r_{2})-m_{1}} \theta(m_{1}+2\delta_{n_{1}}r_{2}, \delta_{n_{1}})(\rho(y))\\
=& u^{\varepsilon_{n_{1}}r_{1}-2\delta_{n_{1}+1} m_{1}-2\delta_{n_{1}}r_{2}} v^{2(1-n_{1})} \ldotp v^{2(n_{1}-1)} u^{2\delta_{n_{1}}r_{2}}  v^{2(1-n_{1})} u^{-2\delta_{n_{1}}r_{2}} \ldotp \\
& u^{2\delta_{n_{1}}r_{2}} x u^{-2\delta_{n_{1}}r_{2}} \ldotp u^{2\delta_{n_{1}}r_{2}} B^{m_1-\delta_{n_1}} u^{-2\delta_{n_{1}} r_{2}} \ldotp u^{2\delta_{n_{1}} r_{2}} \ldotp\\
& (B^{\varepsilon_{n_{1}}} u^{-\varepsilon_{n_{1}}})^{2\varepsilon_{n_{1}}\delta_{n_{1}} r_{2}} B_{0,0}^{\delta_{n_{1}}(1-2r_{2})-m_{1}}\ldotp
\theta(m_{1}+2\delta_{n_{1}}r_{2}, \delta_{n_{1}})(\rho(y))\\
=& u^{\varepsilon_{n_{1}}r_{1}-2\delta_{n_{1}+1} m_{1}-2\delta_{n_{1}}r_{2}} v^{2(1-n_{1})} O_{n_{1}-1,2\delta_{n_{1}}r_{2}} c_{0,2\delta_{n_{1}}r_{2}}(x) B_{0,2\delta_{n_{1}}r_{2}}^{m_1-\delta_{n_1}} \ldotp \\ 
& T_{2\delta_{n_{1}} r_{2},\delta_{n_{1}}} B_{0,0}^{\delta_{n_{1}}(1-2r_{2})-m_{1}} \theta(m_{1}+2\delta_{n_{1}}r_{2}, \delta_{n_{1}})(\rho(y)).
\end{align*}
Applying condition~(\ref{it:algebra_tau_1a}) of Lemma~\ref{lem:algebra_tau_1}, we see that:
\begin{multline*}
c_{2(n_{1}-1),2\delta_{n_{1}+1} m_{1}-\varepsilon_{n_{1}}r_{1}}(y) B_{2(n_{1}-1),2\delta_{n_{1}+1} m_{1}-\varepsilon_{n_{1}}r_{1}}^{r_{2}} J_{n_{1}-1,-2r_{2}} B_{0,0}^{-r_2} \theta(r_2,0)(x)=\\
O_{n_{1}-1,2\delta_{n_{1}}r_{2}} c_{0,2\delta_{n_{1}}r_{2}}(x) B_{0,2\delta_{n_{1}}r_{2}}^{m_1-\delta_{n_1}} T_{2\delta_{n_{1}} r_{2},\delta_{n_{1}}} B_{0,0}^{\delta_{n_{1}}(1-2r_{2})-m_{1}} \theta(m_{1}+2\delta_{n_{1}}r_{2}, \delta_{n_{1}})(\rho(y)),
\end{multline*}
in $\gsigma$, and by Abelianising this equation, we obtain the following equality in $\gsigmab$:
\begin{multline}\label{eq:munu}
(c_{0,2\delta_{n_{1}}r_{2}})_{\ab}(x)-\theta(r_2,0)_{\ab}(x) +\theta(m_{1}+2\delta_{n_{1}}r_{2}, \delta_{n_{1}})_{\ab}\circ\rho_{ab}(y) -(c_{2(n_{1}-1),2\delta_{n_{1}+1} m_{1}-\varepsilon_{n_{1}}r_{1}})_{\ab}(y)=\\\widetilde{J}_{n_{1}-1,-2r_{2}} -\widetilde{O}_{n_{1}-1,2\delta_{n_{1}}r_{2}}  - \widetilde{T}_{2\delta_{n_{1}} r_{2},\delta_{n_{1}}}+r_{2} B_{2(n_{1}-1),2\delta_{n_{1}+1} m_{1}-\varepsilon_{n_{1}}r_{1}}\\
-(m_1-\delta_{n_1})B_{0,2\delta_{n_{1}}r_{2}} -(\delta_{n_{1}}(1-2r_{2})-m_{1}+r_{2})B_{0,0}.
\end{multline}
Now by~(\ref{eq:homo_gsigmab_theta})--(\ref{eq:homo_gsigmab_cpq}), for all $k,l\in \z$, we may check that:
\begin{align}
\mu(B_{k,l})=&(c_{0,2\delta_{n_{1}}r_{2}})_{\ab}(B_{k,l})-\theta(r_2,0)_{\ab}(B_{k,l})\; \text{and}\label{eq:mubkl}\\
\nu(B_{k,l})=&\theta(m_{1}+2\delta_{n_{1}}r_{2}, \delta_{n_{1}})_{\ab}\circ\rho_{ab}(B_{k,l}) \label{eq:nubkl}  -(c_{2(n_{1}-1),2\delta_{n_{1}+1} m_{1}-\varepsilon_{n_{1}}r_{1}})_{\ab}(B_{k,l}).%\notag
\end{align}
Equation~(\ref{eq:muxnuy}) then follows from~(\ref{eq:munu}),~(\ref{eq:mubkl}) and~(\ref{eq:nubkl}).
\end{proof}

In what follows, we suppose that the hypotheses of Proposition~\ref{prop:tau_1_case_6} hold, namely $r_{1}\geq 0$, $r_2 \neq 0$, and either $r_{1}=0$, or $r_{1}>0$ and $e(r_1) > e(r_2)$. With the notation of Lemma~\ref{lem:lema_BU_t2_k2_1_pont_3_a}, we define the homomorphism $\xi_{n_{1},r_2} : \gsigmab \to \ztwo$ on the basis $\{B_{k,l}\}_{k,l\in \z}$ as follows:
\begin{equation}\label{eq:defxi}
\xi_{n_{1},r_2}(B_{k,l})= \begin{cases}
\overline{0} & \text{if $k\neq n_{1}-1$, or if $k= n_{1}-1$ and $2^{e(r_{2})+1} \nmid l$}\\
\overline{1} & \text{if $k= n_{1}-1$ and $2^{e(r_{2})+1} \mid l$.}
\end{cases}
\end{equation}

\begin{lemma}\label{lem:xi_mu_nu}
With the notation of Lemma~\ref{lem:lema_BU_t2_k2_1_pont_3_a}, the compositions $\xi_{n_{1},r_2}\circ \mu: \gsigmab \to \ztwo$ and $\xi_{n_{1},r_2}\circ \nu: \gsigmab \to \ztwo$ are identically zero.
\end{lemma}

\begin{proof}
It suffices to prove that $\xi_{n_{1},r_2}\circ \mu(B_{k,l})= \xi_{n_{1},r_2}\circ \nu(B_{k,l})=\overline{0}$ for all $k,l\in \z$. We start with the case of $\xi_{n_{1},r_2}\circ \mu$. By~(\ref{eq:defmu}), clearly $\xi_{n_{1},r_2}\circ \mu(B_{k,l})=\overline{0}$ if $k\neq n_{1}-1$. So suppose that $k=n_{1}-1$. Using the fact that $\varepsilon_{n_{1}-1}\delta_{n_{1}}=\delta_{n_{1}}$, we have:
\begin{align}
\xi_{n_{1},r_2}\circ \mu(B_{n_{1}-1,l})&= \xi_{n_{1},r_2}(B_{n_{1}-1,l+2\varepsilon_{n_{1}-1}\delta_{n_{1}}r_{2}}-B_{n_{1}-1,l-2\delta_{n_{1}-1}r_{2}})\notag\\
&= \xi_{n_{1},r_2}(B_{n_{1}-1,l+2\delta_{n_{1}}r_{2}})-\xi_{n_{1},r_2}(B_{n_{1}-1,l-2\delta_{n_{1}-1}r_{2}}).\label{eq:calcxi1}
\end{align}
Now $(l+2\delta_{n_{1}}r_{2})-(l-2\delta_{n_{1}-1}r_{2})=2r_{2}(\delta_{n_{1}}+\delta_{n_{1}-1})=2r_{2}$, hence $2^{e(r_{2})+1} \mid l+2\delta_{n_{1}}r_{2}$ if and only if $2^{e(r_{2})+1} \mid l-2\delta_{n_{1}-1}r_{2}$, and it follows from~(\ref{eq:defxi}) and~(\ref{eq:calcxi1}) that $\xi_{n_{1},r_2}\circ \mu(B_{n_{1}-1,l})=\overline{0}$ as required. We now analyse the case of $\xi_{n_{1},r_2}\circ \nu$. Since $-k=n_{1}-1$ if and only if $k+2(n_{1}-1)=n_{1}-1$, it follows from~(\ref{eq:defnu}) that $\xi_{n_{1},r_2}\circ \nu(B_{k,l})=\overline{0}$ if $k\neq -(n_{1}-1)$. So suppose that $k= -(n_{1}-1)$. Since $\varepsilon_{k}\varepsilon_{n_{1}}=\varepsilon_{n_{1}-1}\varepsilon_{n_{1}}=-1$, $\varepsilon_{n_{1}}\varepsilon_{k+1}=\varepsilon_{n_{1}}^{2}=1$, $\delta_{n_{1}-1}\delta_{n_{1}}=0$ and $\varepsilon_{n_{1}-1}\delta_{n_{1}+1}=-\delta_{n_{1}+1}$, we obtain:
\begin{align}
\xi_{n_{1},r_2}\circ \nu(B_{-(n_{1}-1),l})=&-\xi_{n_{1},r_2}(B_{n_{1}-1,l-2\delta_{n_{1}-1}(m_{1}+2\delta_{n_{1}}r_{2})}) -\xi_{n_{1},r_2}(B_{n_{1}-1, l+\varepsilon_{n_{1}-1}(2\delta_{n_{1}+1}m_{1}-\varepsilon_{n_{1}}r_{1})})\notag\\
=&-\xi_{n_{1},r_2}(B_{n_{1}-1,l-2\delta_{n_{1}-1}m_{1}}) -\xi_{n_{1},r_2}(B_{n_{1}-1, l-2\delta_{n_{1}+1}m_{1}+r_{1}}).\label{eq:calcxi2}
\end{align}
Now $(l-2\delta_{n_{1}-1}m_{1})- (l-2\delta_{n_{1}+1}m_{1}+r_{1})=-r_{1}$. Since $2^{e(r_{2})+1} \mid r_{1}$, it follows that $2^{e(r_{2})+1} \mid l-2\delta_{n_{1}-1}m_{1}$ if and only if $2^{e(r_{2})+1} \mid l-2\delta_{n_{1}+1}m_{1}+r_{1}$. Equations~(\ref{eq:defxi}) and~(\ref{eq:calcxi2}) then imply that $\xi_{n_{1},r_2}\circ \nu(B_{-(n_{1}-1),l})=\overline{0}$ as required.
\end{proof}

We now complete the proof of Theorem~\ref{th:BORSUK_TAU_1}. The following remark will be used in the proof of Proposition~\ref{prop:tau_1_case_6}.

\begin{remark}\label{rem:divisible}
Let $r\in \z\setminus \{0\}$, and let $S$ be a set consisting of $2\left\lvert r\right\rvert$ consecutive integers. Then $S$ contains $o(r) = |r|/2^{e(r)}$ elements divisible by $2^{e(r)+1}$. 
\end{remark}

\begin{proof}[of Proposition~\ref{prop:tau_1_case_6}]
Let $\alpha \in [ \torus, * ; \klein,*]$ be a pointed homotopy class such that $\alpha_\#(1,0) =(r_1,2)$ and $\alpha_\#(0,1) =(r_2,0)$, where  $r_2 \neq 0$ and either  $r_1=0$, or $r_1 > 0$ and $e(r_1) > e(r_2)$. Suppose on the contrary that $\alpha$ does not have the Borsuk-Ulam property with respect to $\tau_1$. Then by Lemma~\ref{lem:lema_BU_t2_k2_1_pont_3_a}, there exist $x,y\in \gsigmab$ such that equation~(\ref{eq:muxnuy}) holds. By Lemma~\ref{lem:xi_mu_nu}, $\xi_{n_{1},r_2}(\mu(x)+\nu(y))=\overline{0}$. So to prove the result, it suffices to show that the image of the right-hand side of~(\ref{eq:muxnuy}) by $\xi_{n_{1},r_2}$ is equal to $\overline{1}$. We analyse each of the terms in turn.
\begin{enumerate}[(a)]
\item We start by showing that:
\begin{equation}\label{eq:formJnr}
\xi_{n_{1},r_2}(\widetilde{J}_{n_{1}-1,-2r_{2}})= \delta_{n_{1}+1} \overline{1}.
\end{equation}
If $n_1 = 1$ then $\widetilde{J}_{n_{1}-1,-2r_{2}} = 0$ by Proposition~\ref{prop:words}(\ref{it:wordsb}), so $\xi_{n_{1},r_2}(\widetilde{J}_{n_{1}-1,-2r_{2}})=\overline{0}$. So suppose that $n_1 \neq 1$. By Proposition~\ref{prop:words}(\ref{it:wordsc}) and~(\ref{eq:defxi}), we have:
\begin{equation*}
\xi_{n_{1},r_2}(\widetilde{J}_{n_{1}-1,-2r_{2}})=  \sum_{i=1}^{\sigma_{n_{1}-1}(n_{1}-1)} \; \sum_{j=1}^{\sigma_{r_{2}}2r_{2}} \xi_{n_{1},r_2}  \biggl(B_{\sigma_{n_{1}-1}(2i-1), -\sigma_{2r_{2}}(j-(1-\sigma_{2r_{2}})/2)}\biggr).
\end{equation*}
If $n_{1}$ is odd, there is no integer $i$ satisfying $\sigma_{n_{1}-1}(2i-1)= n_{1}-1$, and thus $\xi_{n_{1},r_2}(\widetilde{J}_{n_{1}-1,-2r_{2}})=\overline{1}$ in this case. If $n_1$ is even, then $\sigma_{n_{1}-1}(2i-1)= n_{1}-1$ if and only if $i=(\sigma_{n_{1}-1}(n_{1}-1)+1)/2$, and in this case, $i$ belongs to the allowed set $\{ 1,\ldots, \sigma_{n_{1}-1}(n_{1}-1) \}$ of indices. Now consider the terms of the form $B_{n_{1}-1, -\sigma_{2r_{2}}(j-(1-\sigma_{2r_{2}})/2)}$, where $j\in \{ 1,\ldots, \sigma_{r_{2}}2r_{2}\}$. Then the set $\{ -\sigma_{2r_{2}}(j-(1-\sigma_{2r_{2}})/2) \mid j= 1,\ldots, \sigma_{r_{2}}2r_{2} \}$ consists of $2\left\lvert r_{2}\right\rvert$ consecutive integers, and thus contains $o(r_{2})$ elements divisible by $2^{e(r_2)+1}$ by Remark~\ref{rem:divisible}. It follows from~(\ref{eq:defxi}) that $\xi_{n_{1},r_2}(\widetilde{J}_{n_{1}-1,-2r_{2}})=\overline{1}$, and this proves~(\ref{eq:formJnr}). 
\item Consider the term $\widetilde{O}_{n_{1}-1,2\delta_{n_{1}}r_{2}}$. If $n_{1}$ is even or is equal to $1$ then $\widetilde{O}_{n_{1}-1,2\delta_{n_{1}}r_{2}}=0$ by Proposition~\ref{prop:words}(\ref{it:wordsb}), and $\xi_{n_{1},r_2}(\widetilde{O}_{n_{1}-1,2\delta_{n_{1}}r_{2}})=\overline{0}$. So assume that $n_{1}$ is odd and different from $1$. By Proposition~\ref{prop:words}(\ref{it:wordsc}) and~(\ref{eq:defxi}), we have:
\begin{align}
\xi_{n_1,r_2}(\widetilde{O}_{n_{1}-1,2\delta_{n_{1}}r_{2}})=&\xi_{n_1,r_2}(\widetilde{O}_{n_{1}-1,2r_{2}})\notag\\
=& \sum_{i=1}^{\sigma_{n_{1}-1}(n_{1}-1)} \; \sum_{j=1}^{\sigma_{r_{2}}2r_{2}} 
\Bigl( \xi_{n_1,r_2} \left( B_{\sigma_{n_1 -1}(2i-1),-\sigma_{2 r_2}j+(\sigma_{2 r_2}-1)/2} \right) +\notag\\
& \xi_{n_1,r_2} \left( B_{\sigma_{n_1 -1}(2i-1)-1,\sigma_{2 r_2}j-(\sigma_{2 r_2}+1)/2} \right) \Bigr).\label{eq:formOnr}
\end{align}
Observe that there is no integer $i$ satisfying $\sigma_{n_1 -1}(2i-1) = n_1 -1$, and it follows that $\xi_{n_1,r_2} \left( B_{\sigma_{n_1 -1}(2i-1),-\sigma_{2 r_2}j+(\sigma_{2 r_2}-1)/2} \right) = \overline{0}$ for all $i\in \{ 1,\ldots,  \sigma_{n_{1}-1}(n_{1}-1)\}$ and $j\in \{ 1,\ldots, \sigma_{r_{2}}2r_{2}\}$. For the second term of~(\ref{eq:formOnr}), note that $\sigma_{n_1 -1}(2i-1) - 1 = n_1 -1$ if and only if $i = (\sigma_{n_{1}-1} (n_1-1) + \sigma_{n_{1}-1}+1)/2$, and in this case, $i$ belongs to the allowed set $\{ 1,\ldots, \sigma_{n_{1}-1}(n_{1}-1) \}$ of indices. Now consider the terms of the form $B_{n_{1}-1, \sigma_{2 r_2}j-(\sigma_{2 r_2}-1)/2}$, where $j\in \{ 1,\ldots, \sigma_{r_{2}}2r_{2}\}$. Then the set $\{ \sigma_{2 r_2}j-(\sigma_{2 r_2}+1)/2 \mid j= 1,\ldots, \sigma_{r_{2}}2r_{2} \}$ consists of $2\left\lvert r_{2}\right\rvert$ consecutive integers, and thus contains $o(r_{2})$ elements divisible by $2^{e(r_2)+1}$ by Remark~\ref{rem:divisible}. It follows from~(\ref{eq:defxi}) that $\xi_{n_{1},r_2}(\widetilde{O}_{n_{1}-1,2r_{2}})=\overline{1}$. Hence:
\begin{equation}\label{eq:formOnr2}
\xi_{n_{1},r_2}(\widetilde{O}_{n_{1}-1,2\delta_{n_{1}}r_{2}})= \delta_{n_{1}} \overline{1}.
%\begin{cases}
%\overline{0} & \text{if $n_{1}$ is even}\\
%\overline{1} & \text{otherwise.}
%\end{cases}
\end{equation}
\item Consider the term $\widetilde{T}_{2\delta_{n_{1}} r_{2},\delta_{n_{1}}}$. If $n_{1}$ is even then $\widetilde{T}_{2\delta_{n_{1}} r_{2},\delta_{n_{1}}}=0$ by Proposition~\ref{prop:words}(\ref{it:wordsb}), and thus $\xi_{n_{1},r_2}(\widetilde{T}_{2\delta_{n_{1}} r_{2},\delta_{n_{1}}})=\overline{0}$. So assume that $n_{1}$ is odd. By Proposition~\ref{prop:words}(\ref{it:wordsc}), we have:
\begin{equation}\label{eq:formTnr2}
\widetilde{T}_{2\delta_{n_{1}} r_{2},\delta_{n_{1}}}=\widetilde{T}_{2r_{2},1}=\sigma_{r_{2}} \sum_{i=1}^{\sigma_{r_{2}}2r_{2}} B_{0,\sigma_{r_{2}}(i-(\sigma_{r_{2}}+1)/2)}.
\end{equation}
If $n_{1}\neq 1$ then $\xi_{n_{1},r_2}(\widetilde{T}_{2r_{2},1})=\overline{0}$ by~(\ref{eq:defxi}). So suppose that $n_{1}=1$. Then the set $\{ \sigma_{r_{2}}(i-(\sigma_{r_{2}}+1)/2) \mid i=1,\ldots, \sigma_{r_{2}}2r_{2}\}$ of indices consists of $2\left\lvert r_{2}\right\rvert$ consecutive integers, and thus contains $o(r_{2})$ elements divisible by $2^{e(r_2)+1}$ by Remark~\ref{rem:divisible}. It follows from~(\ref{eq:defxi}) that $\xi_{n_{1},r_2}(\widetilde{T}_{2r_{2},1})=\overline{1}$. Hence:
\begin{equation}\label{eq:formTnr2!!!V!!!}
\xi_{n_{1},r_2}(\widetilde{T}_{2\delta_{n_{1}} r_{2},\delta_{n_{1}}})=\begin{cases}
\overline{1} & \text{if $n_{1}=1$}\\
\overline{0} & \text{otherwise}.
\end{cases}
\end{equation}
\item Let $\chi=r_{2} B_{2(n_{1}-1),2\delta_{n_{1}+1} m_{1}-\varepsilon_{n_{1}}r_{1}}-(m_1-\delta_{n_1})B_{0,2\delta_{n_{1}}r_{2}} -(\delta_{n_{1}}(1-2r_{2})-m_{1}+r_{2})B_{0,0}$. If $n_{1}\neq 1$ then it follows from~(\ref{eq:defxi}) that $\xi_{n_{1},r_2}(\chi)=\overline{0}$. So suppose that $n_{1}=1$. Then $\chi=r_{2} B_{0,r_{1}}-(m_1-1)B_{0,2r_{2}} -(1-m_{1}-r_{2})B_{0,0}$. By hypothesis, $e(r_1) > e(r_2)$, and we see from~(\ref{eq:defxi}) that $\xi_{n_{1},r_2}(\chi)=\overline{r_{2}}+\overline{m_1-1}+\overline{1-m_{1}-r_{2}}=\overline{0}$. Thus:
\begin{multline}\label{eq:formB00}
\xi_{n_{1},r_2}\Bigl( r_{2} B_{2(n_{1}-1),2\delta_{n_{1}+1} m_{1}-\varepsilon_{n_{1}}r_{1}}-(m_1-\delta_{n_1})B_{0,2\delta_{n_{1}}r_{2}} -\\
(\delta_{n_{1}}(1-2r_{2})-m_{1}+r_{2})B_{0,0}\Bigr)=\overline{0}.
\end{multline}
\end{enumerate}
We now take the image of~(\ref{eq:muxnuy}) by $\xi_{n_{1},r_2}$. Using~(\ref{eq:formB00}) and Lemma~\ref{lem:lema_BU_t2_k2_1_pont_3_a}, it follows that:
\begin{equation}\label{eq:finalmunu}
\overline{0}=\xi_{n_{1},r_2}(\mu(x)+\nu(y))= \xi_{n_{1},r_2}(\widetilde{J}_{n_{1}-1,-2r_{2}} -\widetilde{O}_{n_{1}-1,2\delta_{n_{1}}r_{2}}- \widetilde{T}_{2\delta_{n_{1}} r_{2},\delta_{n_{1}}}).
\end{equation}
From~(\ref{eq:formJnr}),~(\ref{eq:formOnr2}) and~(\ref{eq:formTnr2}), we obtain the following conclusions:
\begin{enumerate}[(i)]
\item if $n_{1}$ is even then $\xi_{n_{1},r_2}(\widetilde{J}_{n_{1}-1,-2r_{2}})=\overline{1}$ and $\xi_{n_{1},r_2}(\widetilde{O}_{n_{1}-1,2\delta_{n_{1}}r_{2}})=  \xi_{n_{1},r_2}(\widetilde{T}_{2\delta_{n_{1}} r_{2},\delta_{n_{1}}})=\overline{0}$.
\item if $n_{1}=1$ then $\xi_{n_{1},r_2}(\widetilde{T}_{2\delta_{n_{1}} r_{2},\delta_{n_{1}}})=\overline{1}$ and $\xi_{n_{1},r_2}(\widetilde{O}_{n_{1}-1,2\delta_{n_{1}}r_{2}})=  \xi_{n_{1},r_2}(\widetilde{J}_{n_{1}-1,-2r_{2}})= \overline{0}$.
\item if $n_{1}$ is odd and $n_{1}\neq 1$ then $\xi_{n_{1},r_2}(\widetilde{J}_{n_{1}-1,-2r_{2}})=\xi_{n_{1},r_2}(\widetilde{T}_{2\delta_{n_{1}} r_{2},\delta_{n_{1}}})= \overline{0}$ and $\xi_{n_{1},r_2}(\widetilde{O}_{n_{1}-1,2\delta_{n_{1}}r_{2}})=\overline{1}$.
\end{enumerate}
In all three cases, we conclude that
$\xi_{n_{1},r_2}(\widetilde{J}_{n_{1}-1,-2r_{2}} -\widetilde{O}_{n_{1}-1,2\delta_{n_{1}}r_{2}}- \widetilde{T}_{2\delta_{n_{1}} r_{2},\delta_{n_{1}}})=\overline{1}$,
which contradicts equation~(\ref{eq:finalmunu}). It follows that $\alpha$ has the Borsuk-Ulam pro\-perty with respect to $\tau_1$.
\end{proof}

%%%%%%%%%%%%%%%%%%%%%%%%%%%%%%%%%%%%%
% Appendix
%%%%%%%%%%%%%%%%%%%%%%%%%%%%%%%%%%%%%

\appendix
\section{Appendix}\label{sec:appendix}

Let $g: F(u,v)\to \zsdz$ be the homomorphism defined on the generators of $F(u,v)$ by $g(u)=(1,0)$ and $g(v)=(0,1)$.

\begin{proposition}\label{prop:basis_gamma}
For each $k,l \in \z$, $l \neq 0$, let $\Gamma_{k,l}=v^k u^l v u^l v^{-k -1} \in F(u,v)$. Then $ \ker g=\left\langle \Gamma_{k,l},\; k,l\in \z,\, l\neq 0 \ | -\right\rangle$.
\end{proposition}
		
\begin{proof}
We use the Reidemeister-Schreier rewriting  process that is described in detail in \cite[Chapter 2, Theorem 2.8]{Magnus} and briefly in \cite[Appendix I, Theorem 6.3]{Mura}. We use the notation of \cite{Mura}. Let $S=\{ v^k u^l \}_{ k,l \in \z }$. We have $g(v^k u^l)= (0,k)(l,0) =((-1)^k l,k)$. So, $g|_S : S \to \zsdz$ is a bijection, and therefore $S$ is a complete set of right coset representatives of $\ker g$ in $F(u,v)$. Moreover, $S$ is a Schreier system of $\ker g$. Let us compute the generators of $\ker g$. We have
\begin{align*}
v^k u^l u \overline{v^k u^l u}^{-1} &=v^k u^{l+1} \overline{v^k u^{l+1}}^{-1}=v^k u^{l+1}(v^k u^{l+1} )^{-1}=\id, \text{ and}\\
v^k u^l v \overline{v^k u^l v }^{-1} &=v^k u^l v(v^{k+1} u^{-l} )^{-1}=\Gamma_{k,l},
\end{align*}
where for all $w\in F(u,v)$, $\overline{w}$ is the unique element of $S$ for which $g(\overline{w})=g(w)$. Note that $\Gamma_{k,l}=\id$ if and only if $l=0$. Using the Reidemeister-Schreier rewriting process, we see that the group $\ker g$ is freely generated by $\{ \Gamma_{k,l} \}_{k,l\in \z,\, l\neq 0}$.
\end{proof}

The basis of $\ker g$ given in Proposition~\ref{prop:basis_gamma} is not well adapted to our calculations. We define a new basis that is more suitable.
		
		\begin{lemma}\label{lem:change_basis}
For each $k,l \in \z$, let $B_{k,l}=v^k u^l B u^{-l} v^{-k}$, where $B=u v u v^{-1}$. Then we have the following relations in $\ker g$:
\begin{enumerate}[(a)]
\item\label{it:basisi} $\Gamma_{k,l} = \prod_{i=1}^{l} B_{k, l-i} $ if $l \geq 1$, and $\Gamma_{k,l} =\prod_{i=1}^{-l} B_{k, l-1+i}^{-1}$ if $l \leq -1$.

\item\label{it:basisii} $B_{k,l}=\Gamma_{k,l+1} \Gamma_{k,l}^{-1}$, where $\Gamma_{k,0}^{-1}=\id$.
		\end{enumerate}
		\end{lemma}

\begin{proof}
We first prove part~(\ref{it:basisi}). We start by proving the result in the case $k=0$. If $l=1$, we have $\Gamma_{0,1}=u v u v^{-1}=B_{0,0}$. So suppose that $\Gamma_{0,l}=\displaystyle \prod_{i=1}^{l} B_{0, l-i}$ for some $l \geq 1$, and let us show that the result holds for $l+1$. We have:
\begin{equation*}
\Gamma_{0,l+1}=u \Gamma_{0,l} u^{-1} \Gamma_{0,1} 
= u \left( \prod_{i=1}^{l} B_{0, l-i}\right) u^{-1} B_{0,0}
= \left( \prod_{i=1}^{l} B_{0, l-i+1} \right) B_{0,0} =  \prod_{i=1}^{l+1} B_{0, (l+1)-i} .
\end{equation*}
By induction, the given formula is valid for $k=0$ and all $l\geq 1$. If $l \leq-1$, the result holds for $\Gamma_{0,-l}$, and thus:
\begin{equation*}
\Gamma_{0,l}=u^l \Gamma_{0,-l}^{-1} u^{-l}
= u^l \left(\displaystyle \prod_{i=1}^{-l} B_{0,-l-i}  \right)^{-1} u^{-l}
= \left(\displaystyle \prod_{i=1}^{-l} B_{0,-i}  \right)^{-1}
= \displaystyle \prod_{i=1}^{-l} B_{0,l-1+i}^{-1}.
\end{equation*}
Hence the formula holds for $k=0$ and all $l\in \z \setminus \{ 0\}$. Now let $k\in \z$. Then $\Gamma_{k,l}=v^k \Gamma_{0,l} v^{-k}$ and $B_{k,l}=v^k B_{0,l} v^{-k}$, and we obtain the formula for all $k\in \z$ and $l\in \z \setminus \{ 0\}$ using the results of the case $k=0$. Part~(\ref{it:basisii}) then follows. 
\end{proof}

\begin{theorem}\label{th:basis_B}
The set $\{ B_{k,l}=v^k u^l B u^{-l} v^{-k} \}_{ k,l \in \z}$ is a basis of $\ker g$.		
\end{theorem}
		
\begin{proof}
By Lemma~\ref{lem:change_basis}, the elements of the set $\{ B_{k,l} \}_{k,l\in \z}$ generate $\ker g$. To show that this set is a basis, it suffices to prove that there are only trivial relations between these elements. Suppose on the contrary that there exists a word $w\in \ker g$ for which:
\begin{equation*}
\text{$w=B_{k_1,l_1}^{\varepsilon1} B_{k_2,l_2}^{\varepsilon_2} \cdots B_{k_n,l_n}^{\varepsilon_n}=\id$,
where $\varepsilon_i \in \{-1,1 \} \text{ and }B_{k_i,r_i}^{\varepsilon_i}B_{k_{i+1},r_{i+1}}^{\varepsilon_{i+1}} \neq \id$.}
\end{equation*}
Let $S=\{k_1, k_2, \ldots, k_n \}$. For each $k \in S$, we  define the set $R_k$ consisting of those indices $l_j$ for which the element $B_{k, l_j}$ 
appears in the word $w$. Let $l_m$ and $l_M$ be the minimal and maximal elements of $R_k$ respectively. We define the sets $B_k$ and $\Gamma_k$ as follows:
\begin{itemize}
\item $B_k=\{ B_{k,0}, \ldots, B_{k,l_m}, \ldots, B_{k,l_M} \}$ and $\Gamma_k=\{ \Gamma_{k,1}. \ldots, \Gamma_{k, l_M}, \Gamma_{k, l_M+1} \}$ if $0 \leq l_m$.
\item $B_k=\{ B_{k, l_m}, \ldots, B_{k,-1}, B_{k,0}, B_{k,1}, \ldots, B_{k,l_M} \}$ and $\Gamma_k=\{ \Gamma_{k, l_m}, \ldots, \Gamma_{k,-1}, \Gamma_{k,1}, \ldots, \Gamma_{k, l_M +1} \}$ if $l_m < 0 < l_M+1$.
\item $B_k=\{ B_{k, l_m}, \ldots,B_{k,l_M} \}$ and $\Gamma_k=\{ \Gamma_{k, l_m}, \ldots, \Gamma_{k, l_M}\}$ if $l_M+1=0$.
\item $B_k=\{ B_{k, l_m}, \ldots,B_{k,l_M}, B_{k,l_M +1}\}$ and $\Gamma_k=\{ \Gamma_{k, l_m}, \ldots, \Gamma_{k, l_M}, \Gamma_{k,l_M +1}\}$ if $l_M+1 < 0$.
\end{itemize}
Note that $B_k$ and $\Gamma_k$ have the same number of elements by Lemma~\ref{lem:change_basis}. Further, if $k, k' \in S$, where $k\neq k'$, then $B_k \cap B_{k'}=\emptyset=\Gamma_k \cap \Gamma_{k'}$. It follows that $C_B=\displaystyle \bigcup_{k \in S} B_s$ and $C_\Gamma=\displaystyle \bigcup_{k \in S} \Gamma_k$ have the same number of elements, and generate the same subgroup $C$ of $\ker g$, using Lemma~\ref{lem:change_basis} once more. Since $C_\Gamma$ is a finite basis of $C$, $C_B$ is a basis of $C$, and so $w \in C$, which yields a contradiction because $C$ is a free group.
\end{proof}

\section*{Acknowledgements}

This work is a continuation of part of the Ph.D. thesis~\cite{Laass} of the third author who was supported by CNPq project nº~140836 and Capes/COFECUB project n\textsuperscript{o}~12693/13-8. The first and second authors wish to thank the `R\'eseau Franco-Br\'esilien en Math\'ematiques' for financial support during their respective visits to the Laboratoire de Math\'ematiques Nicolas Oresme UMR CNRS 6139, Universit\'e de Caen Normandie, from the 9\textsuperscript{th} to the 24\textsuperscript{th} of November 2019, and to the Instituto de Matem\'atica e Estat\'istica, Universidade de S\~ao Paulo, from the 17\textsuperscript{th} to the 31\textsuperscript{th} of August 2019.

%%%%%%%%%%%%%%%%%%%%%%%%%%%%%%%%%%%%%
%%% Bibliography
%%%%%%%%%%%%%%%%%%%%%%%%%%%%%%%%%%%%%

\end{document}